\documentclass[10pt]{amsart}
\usepackage{amssymb}

\theoremstyle{plain}
\newtheorem{prop}{Proposition}[section]
\newtheorem{thm}[prop]{Theorem}
\newtheorem{coro}[prop]{Corollary}
\newtheorem{lemma}[prop]{Lemma}

\theoremstyle{definition}
\newtheorem{defi}[prop]{Definition}
\newtheorem{exam}[prop]{Example}
\newtheorem{rem}[prop]{Remark}

\theoremstyle{remark}

\numberwithin{equation}{section}

\newcommand{\N}{\mathbb N}
\newcommand{\Z}{\mathbb Z}

\newcommand{\R}{\mathbb R}
\newcommand{\C}{\mathbb C}
\newcommand{\Hy}{\mathbb H}

\newcommand{\g}{\gamma}
\newcommand{\G}{\Gamma}
\newcommand{\ld}{\lambda}
\newcommand{\Ld}{\Lambda}

\newcommand{\ba}{\backslash}

\newcommand{\ovr}{\overset}

\newcommand{\f}{\frac}

\newcommand{\I}{\textsl{Id}}
\newcommand{\GL}{\operatorname{GL}}
\newcommand{\Ot}{\operatorname{O}}
\newcommand{\SO}{\operatorname{SO}}
\newcommand{\spec}{\operatorname{Spec}}

\newcommand{\Hom}{\operatorname{Hom}}
\newcommand{\Iso}{\operatorname{Iso}}
\newcommand{\so}{\mathfrak{so}}
\newcommand{\op}{\operatorname}
\newcommand{\tr}{\operatorname{tr}}

\newcommand{\Id}{\textsl{Id}}

\title[Representation equivalence and p-Spectrum]{Representation equivalence and p-Spectrum of constant curvature space forms}
\author{E. A. Lauret, R. J. Miatello and J. P. Rossetti}
\address{FaMAF--CIEM \\ Universidad Nacional de C\'ordoba\\ 5000-C\'ordoba, Argentina.}
\email{elauret@famaf.unc.edu.ar}
\email{miatello@famaf.unc.edu.ar}
\email{rossetti@famaf.unc.edu.ar}
\date{\today}

\begin{document}

\begin{abstract}
We study the $p$-spectrum of a locally symmetric space of constant curvature $\G\ba X$, in connection with the right regular representation of the full isometry group  $G$ of $X$  on  $L^2(\G\ba G)_{\tau_p}$, where $\tau_p$ is the complexified $p$-exterior representation of $\Ot(n)$ on $\bigwedge^p(\R^n)_\C$.
We give an expression of the multiplicity $d_\lambda(p,\G)$ of the eigenvalues of the $p$-Hodge-Laplace operator  in terms of multiplicities $n_\G(\pi)$ of specific irreducible unitary representations of $G$.

As a consequence, we extend results of Pesce for the spectrum on functions to the $p$-spectrum of the Hodge-Laplace operator on $p$-forms of $\G\ba X$, and we compare $p$-isospectrality with $\tau_p$-equivalence for $0\le p\le n$.
For spherical space forms, we show that $\tau$-isospectrality implies $\tau$-equivalence for a class of $\tau$'s that includes the case $\tau=\tau_p$.
Furthermore we prove that $p-1$ and $p+1$-isospectral implies $p$-isospectral.

For nonpositive curvature space forms, we give examples showing that $p$-isospectrality is far from implying $\tau_p$-equivalence, but a variant of Pesce's result remains true. Namely, for each fixed $p$, $q$-isospectrality for every $0\le q\le p$ implies $\tau_q$-equivalence for every $0\le q\le p$. As a byproduct of the methods we obtain several results relating $p$-isospectrality with $\tau_p$-equivalence.
\end{abstract}

\maketitle

\section{Introduction}

Let $X=G/K$ be a homogeneous Riemannian manifold where $G=\Iso(X)$ is the full isometry group and where $K\subset G$ is a compact subgroup.
We shall consider discrete \emph{cocompact} subgroups $\Gamma$ of $G$ acting on $X$ without fixed points, so that $\Gamma\ba X$ is a compact Riemannian manifold.
Under the right regular representation $R_\G$ of  $G$, $L^2(\G\ba G)$ splits as a direct sum
$$
L^2(\G\ba G)= \sum_{\pi\in \widehat G} n_\G(\pi) H_\pi
$$
of closed irreducible subspaces $H_\pi$ with finite multiplicity $n_\G(\pi)$.
Here $\widehat G$ denotes the unitary dual of $G$.
Let $(\tau,V_\tau)$ be a finite dimensional complex unitary representation of $K$ and consider the associated vector bundle
\begin{equation}\label{etau}
E_\tau:=G\underset\tau\times V_\tau \longrightarrow G/K
\end{equation}
endowed with a $G$-invariant inner product (see Subsection~\ref{subsec:bundles}).
Let $\Delta_{\Gamma,\tau}$ denote the Laplace operator acting on sections of the bundle $\Gamma\ba E_\tau\to \Gamma\ba X$ (see Subsection~\ref{subsec:bundles}).

In \cite{Pe2}, Pesce considers spectra of Laplace operators on $\G\ba X$, in connection with the right regular representations $R_{\G,\tau}$ of $G$ on the space
\begin{equation}\label{tauequiv}
L^2(\G\ba G)_\tau := \sum_{\pi\in \widehat G_\tau} n_\G(\pi)
H_\pi,
\end{equation}
where
$\widehat G_\tau=\{\pi\in\widehat G : \operatorname{Hom}_K(\tau,\pi)\ne0\}.$
In the terminology in \cite{Pe2}, two subgroups $\G_1,\G_2$ of $G$, are said to be $\tau$-\emph{representation equivalent} or simply $\tau$-\emph{equivalent}, if the representations $L^2(\G_1\ba G)_\tau$ and $L^2(\G_2\ba G)_\tau$ are equivalent, that is, $n_{\G_1}(\pi)=n_{\G_2}(\pi)$ for any $\pi\in \widehat G_\tau$.
In the case when $\tau={\bf 1}$, the trivial representation of $K$, Pesce calls such groups \emph{$K$-equivalent}.
In analogy, $\Gamma_1\ba X$ and $\Gamma_2\ba X$ are said to be $\tau$-isospectral if the spectra of the Laplace operators $\Delta_{\Gamma_1,\tau}$, $\Delta_{\Gamma_2,\tau}$ are the same.

The question of comparing equivalence (resp.\ $\tau$-equivalence) of representations with isospectrality (resp.\ $\tau$-isospectrality) has been studied by several authors in recent years (see for instance \cite{DG}, \cite{Pe1}, \cite{Pe2}, \cite{GM}, \cite{BR}, \cite{BPR}, \cite{Wo2}).
One has that if two groups $\G_1,\G_2$ are $\tau$-equivalent, then $\Gamma_1\ba X$ and $\Gamma_2\ba X$ are $\tau$-isospectral (see \cite[App.~Prop.~2]{Pe2} or Proposition~\ref{prop2:tau-equiv=>tau-iso}).
Furthermore, Pesce has shown for constant sectional curvature space forms, that the converse holds for $\tau ={\bf 1}$, that is, if the manifolds $\Gamma_1\ba X$ and $\Gamma_2\ba X$ are isospectral on functions, then $\G_1$ and $\G_2$ are $K$-equivalent (see \cite[\S~3, Prop.~2]{Pe2}).

In this paper, again in the  context of spaces of constant sectional curvature, that is, of compact manifolds covered by $S^n$, $\R^n$ or $\Hy^n$, we will study the case when $\tau=\tau_p$, the complexified $p$-exterior representation of $\Ot(n)$ on $\bigwedge^p(\R^n)_\C$, thus
$\Delta_{\Gamma,\tau}$ is the Hodge-Laplace operator acting on $p$-forms.
That is, we study the $p$-spectrum of  $\G\ba X$ in connection with the representation   $L^2(\G\ba G)_{\tau_p}$.
A main tool will be the following formula, valid for any compact locally symmetric space $\G\ba X$ and any representation $\tau$ of $K$, expressing the multiplicity of an eigenvalue $\lambda$ of $\Delta_{\Gamma,\tau}$ in terms of the coefficients $n_\Gamma(\pi)$ for $\pi\in\widehat G_{\tau}$:
\begin{equation*}
d_{\lambda}(\tau, \G)=\sum_{\pi\in\widehat G : \, \ld(C,\pi)=\lambda} n_\G(\pi) \dim\big(\operatorname{Hom}_K(V_\tau^*,H_\pi) \big).
\end{equation*}
Here $\ld(C,\pi)$ denotes a scalar depending only on $\pi$ (see Subsection~\ref{subsec:bundles}).
In the case at hand this formula reduces to
\begin{equation}\label{eq0:d_lambda(tau,Gamma,ld)}
d_{\lambda}(\tau, \G)=\sum_{\pi\in {\widehat G}_{\tau,\ld} } n_\G(\pi).
\end{equation}
where ${\widehat G}_{\tau,\ld}= {\widehat G}_\tau \cap \{\pi \in \widehat G: \lambda(C,\pi) =\lambda \}$.
Therefore,   $\spec_\tau(\G\ba X)$ is determined by the multiplicities $n_\G(\pi)$ for $\pi$ in the sets ${\widehat G}_{\tau,\ld}$.

We will use a general approach that applies to the three cases to be considered.
In light of formula \eqref{eq0:d_lambda(tau,Gamma,ld)}, the goal is to determine the sets ${\widehat G}_{\tau}$, then compute $\ld(C,\pi)$ in each case, and then, for each given $\ld \in \R$, to find the set ${\widehat G}_{\tau,\ld}$. For general $\tau \in \widehat K$ this can be complicated, but it can be carried out for some choices of $\tau$.

As a consequence of the method, by choosing $\tau=\tau_p$, we will give a generalization of results in \cite{Pe2} for the $p$-spectrum of the Hodge-Laplace operator of $\G\ba X$, comparing $p$-isospectrality with $\tau_p$-equivalence. We shall see that, for nonpositive curvature, $p$-isospectrality is far from implying $\tau_p$-equivalence, but a variant of Pesce's result remains true.
We shall consider the three cases: spherical, flat and hyperbolic space forms separately, although they will all share common features.

The case when $X$ has positive curvature has been studied by several authors. Most of the results in this case are included or implicit in the work of
Ikeda-Taniguchi \cite{IkTa}, Ikeda \cite{Ik}, Pesce \cite{Pe1} \cite{Pe2}, Gornet-McGowan \cite{GM} and others.
However, we will give a comprehensive presentation that allows us to extend the results to other choices of $\tau$ (see Proposition~\ref{prop3:mu=2}) and illuminates the cases when the curvature is zero and negative.

\begin{thm}\label{thm0:d_lambda_cpto}
Let $\Gamma$ be a finite subgroup of $\Ot(n+1)$ acting freely on an odd-dimensional sphere $S^{n}$ with $n=2m-1$ and let $0\le p\le n$.

If $\lambda\in\spec_p(\Gamma\ba S^n)$ then $\lambda\in\mathcal E_p \cup \mathcal E_{p+1}$, with $\mathcal E_p$ and $\mathcal E_{p+1}$ disjoint sets, where $\mathcal E_0=\mathcal E_{n+1}=\emptyset$ and for $1\leq p\leq n$,
\begin{equation}
\mathcal E_p = \{\lambda =k^2 + k(n-1) + (p-1)(n -p):  k\in \N\}.
\end{equation}
Furthermore, for each $\lambda \in \mathcal E_p\cup\mathcal E_{p+1}$, we have
\begin{align*}
d_\lambda (p,\G) =
\begin{cases}
    n_\Gamma(\pi_{\Lambda_{k,p},  \delta})  \quad&\text{if }\lambda\in\mathcal E_p,\\
    n_\Gamma(\pi_{\Lambda_{k,p+1}, \delta}) \quad&\text{if }\lambda\in\mathcal E_{p+1}.
\end{cases}
\end{align*}
Here $\pi_{\Lambda_{k,p},\delta}$ is a specific irreducible representation of $\Ot(n+1)$ (see \eqref{eq2:rep_O(2m)_conj}, \eqref{eq2:rep_O(2m)_no_conj})
 where $\pi_{\Lambda_{k,p},\delta}|_{\SO(n+1)}$ has highest weight $\Lambda_{k,p}=k\varepsilon_1+\varepsilon_2+\dots+\varepsilon_p$, 
\begin{equation}\label{eq:klambda}
k=k_{p,\lambda} = -(m-1) + \sqrt{(m-1)^2 + \lambda - (p-1)(n-p)},
\end{equation}
and $\delta\in\{0,\pm1\}$ is uniquely determined by $\lambda$.


In particular, if $\lambda\in\spec_0(\Gamma\ba S^n)$ then $\lambda\in\{k(k+n-1):k\in \N_0\}$ with
\begin{equation}
d_\lambda(0,\G)=n_\G(\pi_{k \varepsilon _1,\delta}),
\end{equation}
where $\pi_{k \varepsilon _1,\delta}$ restricted to $\SO(n+1)$ has highest weight $k\varepsilon_1$ (see Subsection~\ref{subsec:rep_ortho}).
\end{thm}

As a direct consequence:

\begin{coro}\label{coro0:compacto}
Let $\Gamma_1$, $\Gamma_2$ be  finite subgroups of $\Ot(n+1)$ acting freely on $X=S^n$. Then
\begin{enumerate}
\item[(i)] (see  \cite{IkTa}, \cite{Pe2}, \cite{GM}) $\Gamma_1\ba X$ and $\Gamma_2\ba X$ are $p$-isospectral if and only if ${\Gamma_1}$ and ${\Gamma_2}$ are $\tau_p$-equivalent.

\item[(ii)] 
If $\Gamma_1\ba X$ and $\Gamma_2\ba X$ are  $p-1$-isospectral and $p+1$-isospectral, then they are $p$-isospectral.
\end{enumerate}
\end{coro}

In \cite{Ik}, Ikeda constructed for each $p$, lens spaces   $q$-isospectral for every $0 \le q \le p$ but not $p+1$-isospectral.
More recently, Gornet and McGowan \cite{GM}  gave a  very useful survey on the results of Pesce and Ikeda and,
by  computer methods using Ikeda's approach,   exhibited a rich list of lens spaces  that are $p$-isospectral for some values of $p$ only.
Their list (see p.~274) again shows no simple `holes' in the set of values of $p$ for which there is  $p$-isospectrality.
This is consistent with the assertion in (ii) of the corollary that shows that this is valid in general for all spherical space forms.  As noted in \cite{GM}, the examples in \cite{GM} are  $\tau_p$-equivalent for these values of $p$ only.

\medskip

 By following the general method described above we shall prove the following results
for flat and negative curvature compact locally symmetric spaces:

\begin{thm}\label{thm0:d_lambda_flat}
Let $\Gamma$ be a Bieberbach group, that is, $\Gamma$ is a discrete, cocompact subgroup of $\Iso(\R^n)\simeq\Ot(n)\ltimes\R^n$ acting without fixed points on $\R^n$. Let $\Lambda$ denote the translation lattice of $\G$ and let $\Lambda^*$ be the dual lattice of $\Lambda$.
The multiplicity of the eigenvalue $\lambda =4\pi^2\|v\|^2$,   $v \in \Lambda^*$, is given by
\begin{equation}
d_\lambda(\tau_p,\Gamma) =
\begin{cases}
n_\Gamma(\widetilde\tau_p) = \beta_p(\G\backslash \R^n) &\;\text{if }\lambda=0,\\
n_\Gamma(\pi_{\sigma_p,\sqrt\lambda/2\pi})+ n_\Gamma(\pi_{\sigma_{p-1},\sqrt\lambda/2\pi})
    &\;\text{if }\lambda>0.
\end{cases}
\end{equation}
Here   $\sigma_p$ is the $p$-exterior representation of $\Ot(n-1)$ and
$\widetilde\tau_p$ and $\pi_{\sigma_p,r}$ are certain unitary irreducible representations of $\Iso(\R^n)$ (see~\eqref{eq4:rep_E(n)}).
\end{thm}

\begin{thm}\label{thm0:d_lambda_neg}
Let  $G=\SO(n,1)$, $K=\Ot(n)$, $\G\subset G$ be a discrete subgroup acting
without fixed points on $\Hy^n$.

If $0\le p \le n$, and $\ld =0$, then  
\begin{equation*}
d_0(\tau_p,\G) =\beta_p(\G\backslash \Hy^n)=
\begin{cases}
  n_\G\left(J_{\sigma_p, \rho_p}\right) + n_\G\left(J_{\sigma_{p-1},
\rho_{p-1}}\right)
    &\text{if } p\neq\frac{n}{2},\\[1mm]
  n_\G\left(D_\frac n2^+ \oplus D_\frac n2^-\right)
    &\text{if } p=\frac{n}{2}.
\end{cases}
\end{equation*}

If $\lambda\neq0$, then
\begin{equation*}
d_\lambda (\tau_p,\Gamma)=
\begin{cases}
n_\G\left(\pi_{\sigma_p,\sqrt{\rho_{p}^2-\lambda}}\right) +
n_\G\left(\pi_{\sigma_{p-1},\sqrt{\rho_{p-1}^2-\lambda}}\right)
    &\text{if }p\neq\frac{n}{2},\\[2mm]
n_\G\left(\pi_{\sigma_{m},\sqrt{1/4-\lambda}}\right) +
n_\G\left(\pi_{\sigma_{m-1},\sqrt{1/4-\lambda}}\right)
    &\text{if }p=\frac{n}{2}=m.\\
\end{cases}
\end{equation*}
In the expressions above, $\sigma_p$ is the $p$-exterior representation of
$M\simeq\Ot(n-1)$, $\rho_p:=\frac{n-1}2-\min(p,n-1-p)$ and $\pi_{\sigma_p,
\nu}$, $J_{\sigma_p, \nu}$ and $D_\frac n2^+ \oplus D_\frac n2^-$ denote
specific unitary irreducible representations of $G$ (see
Section~\ref{sec:negative}).
 \end{thm}

In the proofs of Theorem~\ref{thm0:d_lambda_flat} and Theorem~\ref{thm0:d_lambda_neg} we use  the description of the unitary duals of $G$ in terms of induced representations.
It will turn out that, generically, there will be at most two irreducible representations in $\widehat G$ contributing to the multiplicity of a given eigenvalue $\ld$ and these multiplicities will be linked to each other for $p$ and $p+1$.
Using this fact, one first shows that 0-isospectrality implies $\tau_0$-equivalence, then one realizes that 0- and 1-isospectrality, taken together, imply $\tau_0$- and $\tau_1$-equivalence, taken together.
In this way, one can build an interval from 0 to $p$ and obtain the assertion in the following theorem that gives a generalization of Pesce's result for nonpositive curvature space forms.

\begin{thm}\label{thm0:A}
Let $X=G/K$ be a simply connected symmetric space of constant nonpositive curvature where $G$ is the full isometry group of $X$.
Let $\G_1, \G_2$ be discrete cocompact subgroups of $G$ acting without fixed points on $X$.
For each $0\leq p\leq n$, $\Gamma_1\ba X$ and $\Gamma_2\ba X$ are $q$-isospectral for every $0\leq q\leq p$ if and only if $\Gamma_1$ and $\Gamma_2$ are $\tau_q$-equivalent for every $0\leq q\leq p$.
\end{thm}

From Theorem~\ref{thm0:A} and its proof, one can derive several consequences relating $p$-isospectrality and $\tau_p$-equivalence (see Proposition~\ref{prop4:other_consequences}, Corollary~\ref{coro4:corolarios} in the flat case and in the negative curvature case). Denote by $\beta_p(M)$ the $p$-th Betti number of $M$. If $X=\R^n$ or $X=\Hy^n$, given  $\G_1, \G_2$ discrete cocompact subgroups of $G=\Iso(X)$ acting without fixed points on $X$, we show
\begin{itemize}
\item If $\G_1, \G_2$ are $\tau_1$-equivalent, then $\Gamma_1\ba X$ and $\Gamma_2\ba X$ are $0$ and $1$-isospectral.

Example \ref{exam4:Klein_bottles} gives two $4$-dimensional compact flat manifolds that are $1$-isospectral but not $0$-isospectral, hence $\Gamma_1, \Gamma_2$ are not $\tau_{1}$-equivalent.

\item If $\G_1, \G_2$ are $\tau_{p+1}$-equivalent (or $\tau_{p-1}$-equivalent) and $\Gamma_1\ba X$ and $\Gamma_2\ba X$ are $p$-isospectral, then $\G_1$ and $\G_2$ are $\tau_{p}$-equivalent.

\item If $\G_1, \G_2$ are $\tau_{p-1}$ and $\tau_{p+1}$-equivalent and $\beta_p(\Gamma_1\ba X) = \beta_p(\Gamma_2\ba X)$, then $\G_1, \G_2$ are $\tau_{p}$-equivalent.
    Hence $\Gamma_1\ba X$ and $\Gamma_2\ba X$ are $p-1$,$p$ and $p+1$-isospectral.

\item If $\Gamma_1\ba X$ and $\Gamma_2\ba X$ are  $p$-isospectral for every $p\in\{1,2,\ldots,k\}$ and they are not $0$-isospectral then $\G_1$ $\G_2$ are not $\tau_{p}$-equivalent for any $p\in \{0,1,2,\ldots,k+1\}$.

In Example~\ref{exam4:chino} we give two flat 8-manifolds that are $p$-isospectral for $p=1,2,3,5,6,7$ but not for $0,4,8$, hence the corresponding groups cannot be $\tau_{p}$-equivalent for any $p\in \{0,1,2,\ldots,8\}$.
Similarly, Example~\ref{exam4:n=4_p_odd_iso} gives two flat 4-manifolds that are $p$-isospectral for $p=1,3$ only.
Thus, these pairs of Bieberbach groups cannot be $\tau_p$-equivalent for any $0\le p \le 4$.
\end{itemize}

The examples we give in the flat case show that, a priori,  the theorems can not be improved substantially.
In the hyperbolic case, similar examples should exist but their construction seems much more
difficult. In general, little is known about the multiplicities $n_\G(\pi)$.

The authors wish to thank Peter Gilkey for several useful comments on a first version of this paper.

\section{General setting and preliminaries}\label{sec:preliminares}
Let $X=G/K$ be a simply connected Riemannian symmetric space, where $G$ is the full isometry group of $X$ and $K$ is the isotropy subgroup of a point in $X$.
Let $\Gamma \subset G$ be a discrete \emph{cocompact} subgroup acting freely on $X$, thus the manifold $\Gamma\backslash X$ inherits a locally $G$-invariant Riemannian structure.
We shall be interested in the cases when $X$ is a space of constant sectional curvature:
\begin{itemize}
\item $X=S^n,\qquad (G,K)=\big(\Ot(n+1),\Ot(n)\big);$
\item $X=\R^n,\qquad(G,K)=\big(\Ot(n)\ltimes \R^n,\Ot(n)\big);$
\item $X=\Hy^n,\qquad (G,K)=\big(\SO(n,1),\Ot(n)\big).$
\end{itemize}
The embedding of $\Ot(n)$ in $\SO(n,1)$ in the third case is the standard one in $\mathrm{S}(\Ot(n)\times\Ot(1))$.

\subsection{Homogeneous vector bundles}\label{subsec:bundles}
Given $(\tau,V_\tau)$, a unitary representation  of $K$, we consider the \emph{homogeneous vector bundle} $E_\tau= G \times_\tau V_\tau$ of $X$.
This is constructed as the quotient of $G\times V_\tau$ under the right action of $K$ given as $(x,v)\cdot k= (xk,\tau (k^{-1})v)$.
We denote $[x,v]$ the class of $(x,v)\in G\times V_\tau$ in $E_\tau$ and $(E_\tau)_{xK}=\{[x,v]\in E_\tau:v\in V_\tau\}$ the fiber of $xK$.
The full isometry group $G$ of $X$ acts on $E_\tau$ by $g[x,v]=[gx,v]$ and sends $(E_\tau)_{xK}$ to $(E_\tau)_{gxK}$ linearly.
We equip  $E_\tau$ with the unique unitary structure which, at the fiber of $eK$, coincides with the unitary structure of $V_\tau$ and such that the action of $G$ is unitary.
This homogeneous vector bundle is \emph{natural} in the sense that an isometry $g$ of $X$ gives an isomorphism of the complex vector spaces $(E_\tau)_{xK}$ and $(E_\tau)_{gxK}$ that preserves the unitary structure.

Let $\Gamma^{\infty}(E_\tau )$ denote the space of smooth sections of $E_\tau$.
Given $\psi \in \Gamma^{\infty}(E_\tau)$, we have that $\psi(xK)= [x,f(x)]$, with $f$ in $C^\infty(G/K;\tau)$, the set of smooth functions $f: G \rightarrow V_\tau$ such that $f(xk)= \tau(k^{-1})f(x)$.
Conversely, any $f\in C^{\infty}(G/K;\tau)$ defines an element $\psi\in\Gamma^{\infty}(E_\tau)$.
The group $G$ acts on $\Gamma^{\infty}(E_\tau)$ on the left by $(g\cdot \psi)(xK):=g\psi(g^{-1} xK)=g[g^{-1}x,f(g^{-1}x)]= [x,f(g^{-1}x)]$, and hence on $C^\infty(G/K;\tau)$ by $(g\cdot f)(x) =f(g^{-1}x)$.

Let $\Gamma$ be a discrete cocompact subgroup of $G$ that acts freely on $X$.
We restrict to $\Gamma$ the left actions of $G$ on $X=G/K$, $E_\tau$, $\Gamma^\infty(E_\tau)$ and $C^\infty(G/K;\tau)$.
The space $\Gamma\backslash X$ is a compact Riemannian manifold and $\Gamma\backslash E_\tau$ is a natural homogeneous vector bundle over $\Gamma\backslash X$.
The space of smooth sections $\Gamma^\infty(\Gamma\ba E_\tau)$ of this vector bundle is isomorphic to the space $C^\infty (\Gamma\ba G/K;\tau)$ of left $\Gamma$-invariant functions in $C^\infty(G/K;\tau)$.
We denote by $L^2(\Gamma\ba E_\tau)$ the closure of $C^\infty (\Gamma\ba G/K;\tau)$ with respect to the inner product
$$
(f_1,f_2)= \int_X \langle f_1(x),f_2(x) \rangle \;\mathrm{d}x.
$$

The Lie algebra $\mathfrak g$ of $G$ acts on $C^\infty(G/K;\tau)$ by
$$
(X\cdot f)(x)=\left.\frac{d}{dt}\right|_{0} f(\exp(-tX)x).
$$
for $X\in\mathfrak g$ and $f\in C^\infty(G/K;\tau)$.
This action induces a representation of the universal enveloping algebra $U(\mathfrak g)$ of $\mathfrak g$.
If $G$ is semisimple we let $C=\sum X_i^2\in U(\mathfrak g)$ where $X_1,\dots,X_n$ is an orthonormal basis of $\mathfrak g$.
In this case, $C$ is called the \emph{Casimir} element.
When $G=\Iso(\R^n)$, thus $X=\R^n$, we let $C=\sum_{i=1}^n X_i^2 \in  U(\mathfrak g)$, where $X_1, \ldots, X_n$ is an orthonormal basis of $\R^n$.
In both cases, the element $C$ does not depend on the basis.

The element $C$ defines a \emph{differential operator} $\Delta_\tau$ on $C^\infty(G/K;\tau)$.
This operator commutes with the left action of $G$ on $C^\infty(G/K;\tau)$, in particular with elements in $\Gamma$, thus $\Delta_\tau$ induces a differential operator $\Delta_{\tau,\Gamma}$ acting on smooth sections of $\Gamma\backslash E_\tau$.

\begin{prop}\label{prop2:Laplace=Casimir}
Let $X=G/K$ be an irreducible simply connected Riemannian symmetric space of constant curvature and  denote by $(\tau_p,\bigwedge^p(\C^n))$ the $p$-exterior representation of $K=\Ot(n)$.
Then $\Delta_{\tau_p,\Gamma}$ coincides with the Hodge-Laplace operator on complex valued differential forms of degree $p$.
\end{prop}

We now recall some notions from the Introduction that will be the main object of this paper.

\begin{defi}
\label{def2-tau_isosp}
Let $\tau$ be a unitary representation of $K$.
Let $\Gamma_1$ and $\Gamma_2$ be two cocompact discrete subgroups of $G$ acting freely on $X$.
The  spaces $\Gamma_1\backslash X$ and $\Gamma_2\backslash X$ are said to be \emph{$\tau$-isospectral} if the Laplace type operators $\Delta_{\tau,\Gamma_1}$ and $\Delta_{\tau,\Gamma_2}$ have the same spectrum.
Here, we shall just say that the spaces are  \emph{$p$-isospectral} if $\tau = \tau_p$.
\end{defi}

Given $\Gamma$ a discrete cocompact subgroup of $G$ acting freely on $X$, we consider the \emph{right regular representation} $R_\Gamma=\operatorname{Ind}_{\Gamma}^G(1_\Gamma)$ of $G$ on $L^{2}(\Gamma\backslash G)$.
This representation decomposes as an orthogonal direct sum of closed invariant subspaces of finite multiplicity
\begin{equation}\label{eq2:L^2(Gamma/G)}
L^2(\Gamma\backslash G)= \sum_{\pi\in \widehat G} n_\G(\pi) \,H_\pi
\end{equation}
where $\widehat G$ is the unitary dual of $G$ and, for each $\pi \in \widehat G$,  $n_\G(\pi)$ denotes the multiplicity of $\pi$ in this decomposition.
Note that if $G$ is noncompact then, generically, $H_\pi$ will be infinite dimensional.

Following the notation in \cite{Pe2}, we let $\widehat G_\tau=\{\pi\in\widehat G: \operatorname{Hom}_K(\tau,\pi)\ne0\}$ and we let $R_{\Gamma,\tau}$ be the unitary subrepresentation of $R_\Gamma $ given by
\begin{equation}\label{eqref{L2}}
L^2(\Gamma\backslash G)_{\tau} = \sum_{\pi\in\widehat G_\tau} n_\Gamma(\pi) H_\pi.
\end{equation}

\begin{defi}(see \cite{Pe2})
Let $\tau$ be an irreducible unitary representation of $K$.
Let $\Gamma_1$ and $\Gamma_2$ be two discrete subgroups of $G$ acting freely on $G/K$.
Then $\Gamma_1$ and $\Gamma_2$ are said to be \emph{$\tau$-equivalent} if the representations $R_{\Gamma_1,\tau}$ and $R_{\Gamma_2,\tau}$ are equivalent, that is, if $n_{\Gamma_1}(\pi) = n_{\Gamma_2}(\pi)$ for every $\pi\in\widehat G_\tau$.
\end{defi}

\begin{prop}\label{prop2:multip_lambda}
If $\ld\in\R$, the multiplicity $d_{\ld}(\tau, \G)$ of the eigenvalue $\lambda$ of
$\Delta_{\tau,\Gamma}$ is given by
\begin{equation}\label{multld}
d_{\ld}(\tau, \G)=
\sum_{\pi\in \widehat G :\, \lambda(C,\pi)=\ld}
n_\G(\pi) \dim\left(\operatorname{Hom}_K (V_\tau^*,H_\pi) \right).
\end{equation}
\end{prop}

\begin{proof}
This result is well-known.
We sketch the proof for completeness.
One has a map $\phi:C^\infty(\Gamma \ba G) \times V_\tau \longrightarrow C^\infty (\Gamma\ba G,V_\tau)$
given by $\phi(f,v)=f(g)v$.
Thus $\phi$ induces a homomorphism $\overline\phi:C^\infty(\Gamma \ba G) \otimes V_\tau \rightarrow C^\infty (\Gamma\ba G,V_\tau)$ that is actually an isomorphism and preserves the $K$-action. Indeed
\begin{align*}
\phi(R_k f,\tau(k)v)(g)
    &= f(gk)\tau(k)(v) = \tau(k) f(gk)v\\
    &= \tau(k) \phi(f,v) (gk)= \big(k\cdot \phi(f,v) \big)(g).
\end{align*}
Hence $\overline\phi$ sends $K$-invariants isomorphically onto $K$-invariants, thus
$$
\left(C^\infty(\Gamma \ba G) \times V_\tau \right)^K
    \simeq C^\infty (\Gamma\ba G,V_\tau)^K= C^\infty(\Gamma\ba G/K;\tau)\simeq \G^\infty(\G\ba E_\tau)
$$
Now
$$
\left(C^\infty(\Gamma \ba G) \times V_\tau \right)^K
    = \sum_{\pi\in\widehat G} n_\Gamma(\pi) \,\left(H_\pi^\infty\otimes V_\tau\right)^K
    \simeq \sum_{\pi\in\widehat G} n_\Gamma(\pi) \,\Hom_K\left(V_\tau^*,H_\pi^\infty\right).
$$
Thus
$$
L^2(\Gamma\ba E_\tau)_\lambda \simeq
    \sum_{\pi\in \widehat G :\, \lambda(C,\pi)=\ld}
    n_\G(\pi) \left(\operatorname{Hom}_K (V_\tau^*,H_\pi) \right).
$$
\end{proof}

From formula~\eqref{multld} one sees that the only representations in $\widehat G$ that can contribute to the multiplicity of the eigenvalue $\ld$ are those in $\widehat G_\tau$.
As a direct consequence we have that:

\begin{prop}\label{prop2:tau-equiv=>tau-iso}
Let $\Gamma_1$ and $\Gamma_2$ be discrete cocompact subgroups of $G$ acting freely on $X$.
If $\Gamma_1$ and $\Gamma_2$ are $\tau$-equivalent then $\Gamma_1\ba X$ and $\Gamma_2\ba X$ are $\tau$-isospectral.
\end{prop}

\subsection{Unitary dual group of the orthogonal group}\label{subsec:rep_ortho}
If $X$ is a symmetric space of constant curvature, then either $X=S^n$, $X=\R^n$ or $X=\Hy^n$.
In all three cases we have  $K \simeq\Ot(n)$.
We will need some well known facts about the irreducible representations of $\Ot(n)$.

We first recall the root system of the complex simple Lie algebra $\so(n,\C)$.
Let
$$
\mathfrak h=\left\{H=\sum_{j=1}^m i h_j(E_{2j-1,2j} - E_{2j,2j-1}): h_j\in\C\right\}.
$$
Then $\mathfrak h$ is a Cartan subalgebra of $\so(2m,\C)$ and also of $\so(2m+1,\C)$ if we add a zero row and a zero column at the end.
For $H\in\mathfrak h$, set $\varepsilon_j(H)=h_j$ for $1\leq j\leq m$.
We consider the inner product $\langle\,,\,\rangle$ on $\mathfrak h_{\R}$ obtained by $\frac{1}{2(n-1)}$ times the restriction of the Killing form on $\mathfrak g$, and its dual form on $h_{\R}^*$.
The root systems of $\so(2m+1,\C)$ and $\so(2m,\C)$ with respect to $\mathfrak h$ and $\langle\,,\,\rangle$ are of type $B_m$ and $D_m$ respectively.
We list the roots in Table~\ref{tabla:rootsystem}.

\begin{table}[b]
\caption{Root systems for $\so(n)$.}\label{tabla:rootsystem}
$
\begin{array}{c|l|l}
&\multicolumn{1}{c|}{\so(2m+1)} & \multicolumn{1}{c}{\so(2m)}\\ \hline
\text{roots} &
    \begin{array}{l}
    \pm\varepsilon_i\pm\varepsilon_j\quad (i\neq j) \\ \pm\varepsilon_i
    \end{array} &
    \begin{array}{l}
    \pm\varepsilon_i\pm\varepsilon_j\quad (i\neq j)
    \end{array}
\\ \hline
\text{\begin{tabular}{c}positive \\ roots\end{tabular}}&
    \begin{array}{l}
    \varepsilon_i\pm\varepsilon_j\quad (i<j) \\ \varepsilon_i
    \end{array}&
    \begin{array}{l}
    \varepsilon_i\pm\varepsilon_j\quad (i<j)
    \end{array}
\\ \hline
\text{\begin{tabular}{c}simple \\ roots\end{tabular}}&
    \begin{array}{l}
    \varepsilon_i-\varepsilon_{i+1} \quad (1\leq i<m) \\ \varepsilon_m
    \end{array}&
    \begin{array}{l}
    \varepsilon_i-\varepsilon_{i+1}\quad (1\leq i<m) \\ \varepsilon_{m-1}+\varepsilon_m
    \end{array}
 \end{array}
$
\end{table}

The finite-dimensional irreducible representations of a complex simple Lie algebra are characterized by their corresponding highest weights.
We will denote them by $\mathcal{P}(\mathfrak g)$.

We have
\begin{align*}
\mathcal{P}(\so(2m))
    &= \left\{\sum_{i=1}^m c_i\varepsilon_i:
        \begin{array}{l}
          c_i\in\Z\;\forall i \text{ or }c_i\in\frac12+\Z\; \forall i\text{, and}\\
          c_1\geq c_2\geq\dots\geq c_{m-1}\geq| c_m|.
        \end{array}\right\},
\\
\mathcal{P}(\so(2m+1))
    &= \left\{\sum_{i=1}^m  c_i\varepsilon_i:
        \begin{array}{l}
           c_i\in\Z\;\forall i \text{ or } c_i\in\frac12+\Z\; \forall i\text{, and}\\
           c_1\geq c_2\geq\dots\geq c_{m-1}\geq c_m\geq0.
        \end{array}\right\}.
\end{align*}

The irreducible representations of $\so(n)$ are in a one to one correspondence with those of the simply connected Lie group $\mathrm{Spin}(n)$.
In the case of  $\SO(n)$, the highest weights of the  irreducible representations  are given by
\begin{equation}\label{eq2:P(SO(n))}
\mathcal{P}(\SO(n))
    = \left\{\sum_{i=1}^m  c_i\varepsilon_i\in \mathcal{P}(\so(n)):  c_i\in\Z\;\forall i\right\}.
\end{equation}

\begin{exam}\label{exam2:tau_p-SO(n)}
Set $\Lambda_p=\sum_{i=1}^{\min(p,n-p)}\varepsilon_i\in\mathcal{P}(\SO(n))$ for $0\leq p\leq n$.
If $p\neq \frac n2$, then $\Lambda_p$ is the highest weight of the $p$-exterior representation on $\bigwedge \C^n$ of $\SO(n)$.
These representations are irreducible.
The $m$-exterior power representation $\bigwedge^{m}(\C^{2m})$ of $\SO(2m)$ decomposes  as $\bigwedge_+^{m}(\C^{2m})\oplus \bigwedge_-^{m}(\C^{2m})$, where $\bigwedge_{\pm}^{m}(\C^{2m})$ are irreducible and have highest weights $\sum_{j=1}^{m-1}\varepsilon_j\pm\varepsilon_m$.
\end{exam}

We now describe the irreducible regular representations of the full orthogonal
group $\Ot(n)$ in terms of the irreducible representations of the special orthogonal group $\SO(n)$.
Let
\begin{equation}\label{eq2:g_0}
g_0=
\begin{cases}
-\I_{n} &\quad
    \text{if $n$ is odd,}\\
\left[\begin{smallmatrix}
\I_{n-1}&\\ &-1
\end{smallmatrix}\right]&\quad
    \text{if $n$ is even.}
\end{cases}
\end{equation}
Then $\Ot(n)=\SO(n)\cup g_0\SO(n)$,  thus we will define the representations of $\Ot(n)$ on each component, $\SO(n)$ and $g_0\SO(n)$.

For $\Lambda\in \mathcal {P}(\SO(2m+1))$ and $\delta=\pm1$, let $(\pi_\Lambda,V)$ be the representation of ${\SO}(2m+1)$ with highest weight $\Lambda$.
Then we may define a representation $(\pi_{\Lambda,\delta},V)$ of ${\Ot}(2m+1)$ on $V$  by setting, for $g \in {\Ot}(2m+1)$,
\begin{equation}\label{eq2:rep_O(2m+1)}
\pi_{\Lambda,\delta}(g)(v)=
    \begin{cases}
      \pi_\Lambda(g)(v)     &\;\text{if }g\in\SO(2m+1),\\
      \delta\,\pi_\Lambda( g_0g)( v) &\;\text{if }g\in g_0\SO(2m+1).
    \end{cases}
\end{equation}

For $\Lambda=\sum_{j=1}^m c_j\varepsilon_j\in \mathcal {P}(\SO(2m))$ ($ c_j\in\Z$ for all $j$ and $ c_1\geq\dots\geq c_{m-1}\geq| c_m|$), we denote by $\overline\Lambda=\sum_{j=1}^{m-1} c_j\varepsilon_j- c_m\varepsilon_m\in \mathcal {P}(\SO(2m))$.
Let $(\pi_\Lambda,V_\Lambda)$ be the irreducible representation of $\SO(2n)$ with highest weight $\Lambda$.
If $I_{g_{0}}(g)=g_0gg_0$, then $I_{g_{0}}$ defines an automorphism of $\SO(2m)$ and
one can see that $(\pi_\Lambda \circ I_{g_{0}},V_\Lambda)$ has highest weight $\overline\Lambda$. Thus, there exists a unitary operator $T_\Lambda:V_\Lambda\to V_{\overline\Lambda}$ such that $T_\Lambda\circ(\pi_\Lambda\circ I_{g_{0}})(g)=\pi_{\overline\Lambda}(g)\circ T_\Lambda$ for every $g\in \SO(m)$. Furthermore, $(\pi_\Lambda \circ I_{g_{0}},V_\Lambda)$ is equivalent to $(\pi_\Lambda,V_\Lambda)$ if and only if $ c_m=0$.

If $\Lambda\in \mathcal {P}(\SO(2m))$ is such that $ c_m=0$ and $\delta\in\{\pm1\}$, we define a representation $\pi_{\Lambda,\delta}$ of $\Ot(2m)$ on $V_\Lambda$ as
\begin{equation}\label{eq2:rep_O(2m)_conj}
\pi_{\Lambda,\delta}(g)(v)=
\begin{cases}
    \pi_\Lambda(g)(v),&\quad\text{if $g\in\SO(2m)$,}\\
    \delta\, T_\Lambda (\pi_\Lambda(g_0g)(v)), &\quad\text{if $g\in g_0 \SO(2m)$.}
\end{cases}
\end{equation}
Note that this definition depends on the choice of $T_\Lambda$ since $-T_\Lambda$ is another intertwining operator between $\pi_\Lambda$ and $\pi_{\overline\Lambda}$.
However, we have $\pi_{\Lambda,\delta}\simeq\pi_{\Lambda,-\delta}\otimes\det$.

If $\Lambda\in \mathcal {P}(\SO(2m))$ is such that $ c_m>0$, we  set $\delta=0$ and
define the representation $\pi_{\Lambda,0}$ of $\Ot(2m)$ on $V_\Lambda\oplus V_{\overline\Lambda}$ as follows
\begin{equation}\label{eq2:rep_O(2m)_no_conj}
\pi_{\Lambda,0}(g)(v,v')=
\begin{cases}
    (\pi_{\Lambda}(g)(v),\pi_{\overline\Lambda}(g)(v')),&\quad
        \text{if $g\in\SO(2m)$}\\
    (\pi_\Lambda(g_0g)v',\pi_\Lambda(g_0g)v)&\quad\text{if $g \in g_0\SO(2m)$.}
\end{cases}
\end{equation}
In particular $\pi_{\Lambda,0}(g_0)(v,v')=(v',v)$, thus  $(V_\Lambda\oplus V_{\overline\Lambda}, \pi_{\Lambda,0})$ is irreducible.

In the next theorem we describe the unitary dual of $G=\Ot(n)$.

\begin{thm}\label{prop2:dual_O(n)}
We have
\begin{eqnarray*}
\widehat{\Ot(2m+1)}
    &=&
    \left\{\pi_{\Lambda,\delta} \textrm{ as in }\eqref{eq2:rep_O(2m+1)} : \Lambda\in \mathcal{P}(\SO(2m+1)), \, \delta =\pm 1\right\},\\
\widehat{\Ot(2m)}
    &=&
    \left\{\pi_{\Lambda,\delta} \textrm{ as in }\eqref{eq2:rep_O(2m)_conj} : \Lambda \in \mathcal{P}(\SO(2m)), \, c_m=0,\, \delta =\pm1 \right \}\\
    & &\quad
    \cup\;\left\{\pi_{\Lambda,0} \textrm{ as in }\eqref{eq2:rep_O(2m)_no_conj} : \Lambda \in \mathcal{P}(\SO(2m)), \, c_m>0 \right \}.
\end{eqnarray*}
\end{thm}

\begin{exam}\label{exam2:tau_p-O(m)}
We denote by $(\tau_p,\bigwedge^p(\C^{n}))$  the complexification of the $p$-exterior representation of the canonical representation of $\Ot(n)$ on $\R^n$.
We have that  $\tau_p$  is irreducible for every value of $p$.
Furthermore, $\tau_p \simeq \tau_{n-p}\otimes \det$ for any $0\le p \le n$, where the intertwining operator is given by the Hodge star operator.

Recall that $\Lambda_p=\sum_{i=1}^{\min(p,n-p)}\varepsilon_i\in\mathcal{P}(\SO(n))$.
In the notation of Theorem~\ref{prop2:dual_O(n)}, if $n$ is odd we have $\tau_p \simeq \pi_{\Lambda_p,(-1)^p}$ and, for $n$ even, $\tau_{\frac n2}\simeq\pi_{\Lambda_{\frac n2},0}$.
To write $\tau_p\in\SO(2m)$ as \eqref{eq2:rep_O(2m)_conj} for $p\neq m$, we must fix an intertwining operator $T_{\Lambda_p}$.
For $0\leq p<m$, we write $\bigwedge^p(\C^{2m})= W_0\oplus W_1$, where $W_0$ (resp.\ $W_1$) is the subspace of $\bigwedge^p(\C^{2m})$ generated by $e_{i_1}\wedge\dots\wedge e_{i_p}$ where $2m\notin\{i_1,\dots,i_p\}$ (resp.\ $2m\in\{i_1,\dots,i_p\}$).
It is not hard to check that $T_{\Lambda_p}:= \Id_{W_0}\oplus (-\Id_{W_1})$ satisfies
$T_{\Lambda_p}\circ(\pi_{\Lambda_p}\circ I_{g_{0}})(g)=\pi_{\overline{\Lambda}_p}(g)\circ T_{\Lambda_p}$ for every $g\in\SO(2m)$.
Finally, one has that $\tau_p \simeq \pi_{\Lambda_p,1}$ for $0\leq p<m$ and $\tau_p \simeq \pi_{\Lambda_p,-1}$ for $m<p\leq n$.
\end{exam}

We conclude this section by stating two branching laws for orthogonal groups that will be needed in the following sections.

\begin{prop}\label{prop2:tau_p|_O(n-1)}
Let $\tau_p$ and $\sigma_p$ be the $p$-exterior representations of $\Ot(n)$ and $\Ot(n-1)$ respectively.
Then, for any $0\le p\le n$, we have
\begin{eqnarray}
\tau_p|_{\Ot(n-1)} &=&
      \sigma_p \oplus \sigma_{p-1},
\end{eqnarray}
with the understanding that $\sigma_{-1}$, $\sigma_n$ are the zero representations of $\Ot(n-1)$.
That is, $\tau_0|_{\Ot(n-1)}=\sigma_0$ and $\tau_n|_{\Ot(n-1)}=\sigma_{n-1}$.

\end{prop}

\begin{lemma}\label{lem2:tau_p<pi|_O(2m-1)}
Let $\tau_p$ be the $p$-exterior representation of $\Ot(2m-1)$ and let $\pi_{\Lambda,\delta}\in\widehat{\Ot(2m)}$ in the notation of Theorem~\ref{prop2:dual_O(n)}.
Then $[\tau_p:\pi_{\Lambda,\delta}|_K]>0$ if and only if
\begin{equation}\label{lambdataup}
\Lambda= k\varepsilon_1 +\varepsilon_2+\dots+\varepsilon_p+c_{p+1}\varepsilon_{p+1}
\end{equation}
with $k\in\N$, $c_{p+1}\in\{0,1\}$, and where $\delta\in\{0,\pm1\}$ has a specific value.
More precisely, if $p=m-1, m$ and $c_m>0$ then $\delta=0$ while if $p\ne m-1,m$ or $p= m-1,m$ and $c_m=0$, then $\delta =\pm 1$ and the sign depends on $p$ and on the choice of the intertwining operator $T_\Lambda$.
Moreover $[\tau_p:\pi_{\Lambda,\delta}|_K]=1$.
\end{lemma}

\section{Compact case}\label{sec:compact}

In this section we shall prove the assertions in Theorem~\ref{thm0:d_lambda_cpto} and Corollary~\ref{coro0:compacto} for constant curvature spaces of compact type, that is, for spherical space forms.
We fix the following notation for this section:
\begin{align*}
G&= \Ot(n+1)\simeq \Iso(S^n) ,\\
K&= \Ot(n)= \left\{g\in G : g.e_{n+1} =e_{n+1}\right \},\\
X&= G/K \simeq S^n.
\end{align*}
Note that, in all three cases, $G$ and $K$ have two connected components.
Since even dimensional spheres $S^{n}$ cover only $S^n$ and $\R P^{n}$, and their spectra are well-known, we will look only at odd dimensional spheres.
Thus, we assume throughout this section that $n=2m-1$, then $G=\Ot(2m)$ and $K=\Ot(2m-1)$.
We first describe the set $\widehat G_{\tau_p}$, in the notation of Theorem~\ref{prop2:dual_O(n)}.
Set, for $2\leq p\leq n-2$ and $k\in\N$,
\begin{equation}\label{eq:Lambdakp}
\Lambda_{k,p} = k\varepsilon_1+\varepsilon_2+\dots+\varepsilon_{\min(p,n-p)}
\end{equation}
and $\Lambda_{k,p}=k\varepsilon_1$ for $p=1,n-1$ and $k\in\N_0$.
In particular, $\Lambda_{1,p}=\Lambda_p$, as in Example~\ref{exam2:tau_p-SO(n)}.

\begin{prop}\label{prop3:G_tau_p}
Let $\tau_p$ be the $p$-exterior representation of $K$.
If $0<p<m-1$, then
\begin{align}
\widehat G_{\tau_0} &=
  \left\{ \pi_{\Lambda_{0,1},\,\delta}, \pi_{\Lambda_{k,1},\,\delta}: k\in \N \;\text{with } \delta \in \{\pm1\}\right\},
\notag \\
\widehat G_{\tau_p} &=
    \left\{ \pi_{\Lambda_{k,p},\,\delta},\,\pi_{\Lambda_{k,p+1},\,\delta} : k\in \N \;\text{with } \delta \in \{\pm1\} \right \},
\notag\\
\widehat G_{\tau_{m-1}} &=
    \left\{ \pi_{\Lambda_{k,m-1},\,\delta},\, \pi_{\Lambda_{k,m},\,0}: k\in \N \;\text{with } \delta \in \{\pm1\} \right\}.
\notag
\end{align}
Furthermore, if $m\leq p\leq 2m-1=n$, then $
\widehat G_{\tau_p} =
\left\{\pi_{\Lambda,\delta} : \pi_{\Lambda,-\delta}\in \widehat G_{\tau_{n-p}} \right\}.
$
In the sets above, $\delta$ is uniquely determined by $k$, $p$ and $T_\Ld$.
Moreover, for any $0\leq p<n$ and $k\in\N$, $\pi_{\Lambda_{k,p+1},\,\delta} \in \widehat G_{\tau_p} \cap \widehat G_{\tau_{p+1}}$.
\end{prop}

\begin{proof}
From Theorem~\ref{prop2:dual_O(n)} we see that $\widehat{\Ot(2m)}$ is the set of all representations $\pi_{\Lambda,\delta}$ where $\Lambda=\sum_{i=1}^m c_i\varepsilon_i\in \mathcal {P}(\SO(2m))$ (see~\eqref{eq2:P(SO(n))}), $ c_m\in\N_0$ and either $\delta=\pm1$ if $ c_m=0,$  or $\delta=0$ if $ c_m\in\N.$ Also, from Example~\ref{exam2:tau_p-O(m)} we see that, if $p>0$, $\tau_p=\tau_{\Lambda_p,\kappa}\in\widehat K$ as in \eqref{eq2:rep_O(2m+1)}, where $\Lambda_p=\sum_{j=1}^p\varepsilon_j$ and $\kappa=(-1)^p$.

Taking this into account,  by using the branching law  in Lemma~\ref{lem2:tau_p<pi|_O(2m-1)} one checks that the description of $\widehat G_{\tau_p}$ is as stated in the proposition.
\end{proof}

Now we prove 
that, for a spherical space form, the multiplicity of each eigenvalue of the Hodge-Laplace operator on $p$-forms involves  only one specific $n_\Gamma(\pi)$, that is to say,  the sum in \eqref{multld} has only one term.
We set $\mathcal E_0=\mathcal E_{n+1}=\emptyset$,
\begin{equation}
\begin{array}{r@{\,}l}
\mathcal E_1=\mathcal E_{n}&= \{k+k(n-1): k\in\N_0\},\quad\text{and} \\ [1mm]
\mathcal E_p&= \{k+k(n-1)+(p-1)(n-p): k\in\N\}
\end{array}
\end{equation}
for $1< p < n$.
Note that $\mathcal E_{p}=\mathcal E_{n+1-p}$ for every $0\leq p\leq n+1$.

\begin{proof}[Proof of Theorem~\ref{thm0:d_lambda_cpto}]
By Schur's lemma, the Casimir element $C$ acts on any irreducible representation $\pi_{\Lambda,\delta}$ with $\Lambda= \sum_{i=1}^{m} c_i\varepsilon_i$ by multiplication by a scalar $\lambda(C,\pi)$ given by
\begin{equation}\label{eq3:Casimir_on_pi}
\lambda(C,\pi_{\Lambda,\delta})=\langle\Lambda+\rho, \Lambda+\rho\rangle- \langle\rho,\rho\rangle = \langle\Lambda, \Lambda\rangle + 2\langle\Lambda,\rho\rangle\,,
\end{equation}
where $\rho= \sum_{i=1}^{m}(m-i)\varepsilon_i$.
Note that the scalar $\lambda(C,\pi_{\Lambda,\delta})$ does not depend on $\delta$.

We first assume that $p=0$.
By Proposition~\ref{prop3:G_tau_p}, the highest weights of representations in $\widehat G_{\tau_0}$ have the form $\Lambda= k\varepsilon_1$ with $k\in \N_0$ and
\begin{equation}\label{eq3:Casimir_on_pi}
\lambda(C,\pi_{k\varepsilon_1,\delta})= k^2 + 2k(m-1)=k(k+n-1).
\end{equation}
Proposition~\ref{prop2:multip_lambda} now implies that if $\lambda\notin \mathcal E_1$ then $\lambda$ is not in $\spec_0(\Gamma\ba S^n)$, that is, $d_\lambda(\tau_p,\Gamma)=0$.
Moreover, since $k \mapsto k(k+n-1)$ is increasing for $k\ge 0$ hence $k=k_\lambda$ is clearly determined by $\lambda\in\mathcal E_1$.
Actually
\begin{equation}
k_\lambda= -\tfrac{n-1}2 + \sqrt{(\tfrac{n-1}2)^2 + \lambda}= -(m-1) + \sqrt{(m-1)^2 + \lambda}.
\end{equation}
Thus, in this case $d_\lambda(\tau_0,\G)=n_\G(\pi_{\Lambda_{k_\lambda\,\varepsilon_1},\delta})$.

\medskip

Now assume  $0<p<m$.
By Proposition~\ref{prop3:G_tau_p}, if $\pi_{\Lambda,\delta} \in \widehat G_{\tau_p}$ then
$\Lambda =\Lambda_{k,p}$ or $\Lambda =\Lambda_{k,p+1}$ and by \eqref{multld}, for each $\lambda$, we must consider $\pi_{\Lambda_{k,p},\delta}, \pi_{\Lambda_{k,p+1},\delta}  \in \widehat G_{\tau_p}$ with $\lambda(C,\pi_{\Lambda_{k,p},\delta})= \ld$ or $\lambda(C,\pi_{\Lambda_{k,p+1},\delta})=\ld$ .
Since
\begin{align*}
\lambda(C,\pi_{\Lambda_{k,p},\delta})
    &= \langle\Lambda_{k,p}, \Lambda_{k,p}\rangle + 2\langle\Lambda_{k,p},\rho\rangle\\
    &= k^2 +p-1 + 2k(m-1)+ \textstyle\sum\limits_{i=2}^p (m-i)\varepsilon_i\\
    &= k^2 + k(n-1) + (p-1)(n-p)
\end{align*}
lies in $\mathcal E_p$ and $\lambda(C,\pi_{\Lambda_{k,p+1},\delta})= k^2 + k(n-1) + p(n -p-1)\in\mathcal E_{p+1}$, it follows that $\lambda$ is not an eigenvalue of $\Delta_{\tau,\Gamma}$ if $\lambda\notin\mathcal E_p\cup\mathcal E_{p+1}$.

It is clear that for $\lambda\in \mathcal E_p$ or $\lambda\in\mathcal E_{p+1}$, $k$ is uniquely determined by $\lambda$.
Indeed, we have
$$
k_\lambda =
\begin{cases}
-(m-1) +\sqrt{(m-1)^2 + \lambda - (p-1)(n-p)} &\quad\text{if $\lambda\in\mathcal E_p$,}\\[1mm]
-(m-1) +\sqrt{(m-1)^2 + \lambda - p(n-p-1)} &\quad\text{if $\lambda\in\mathcal E_{p+1}$.}
\end{cases}
$$

It remains only to check that  $\mathcal E_p$ and $\mathcal E_{p+1}$ are disjoint.
Thus, let us assume that $\lambda(C,\pi_{\Lambda_{k,p},\delta}) = \lambda(C,\pi_{\Lambda_{h,p+1},\delta})$ for some $h,k\in\N$. Then
$$
k^2 + k(n-1) + (p-1)(n-p) = h^2 + h(n-1) + p(n-p-1),
$$
which implies that
\begin{equation}\label{eq:equalevalues}
(k- h)(k+h+n-1) = n-2p.
\end{equation}
Now since $n>2p$, then $k>h$, thus the left-hand side is at least $n+1 >n-2p$, hence \eqref{eq:equalevalues} cannot hold.
Thus, $\mathcal E_p\cap\mathcal E_{p+1}=\emptyset$ for $0\leq p\leq n+1$.

It follows from this that for each $\lambda\in \mathcal E_p\cup \mathcal E_{p+1}$, the sum in Proposition~\ref{prop3:G_tau_p} has only one term, indeed
\begin{align*}
d_\lambda (\tau_p,\Gamma) =
\begin{cases}
n_\Gamma(\pi _{\Lambda_{{k_\lambda,p}}, \delta})
    &\quad\text{if $\lambda\in\mathcal E_p$,}\\
n_{\Gamma}(\pi _{\Lambda_{{k_{\lambda},p+1}}, \delta})
    &\quad\text{if $\lambda\in\mathcal E_{p+1}$.}
\end{cases}
\end{align*}

The case  when $m\leq p\leq 2m-1$ follows in the same way since $\pi_{\Lambda,\delta}\in\widehat G_{\tau_p}$ if and only if $\pi_{\Lambda,-\delta}\in\widehat G_{\tau_{n-p}}$ by Proposition~\ref{prop3:G_tau_p} and $C$ acts by the same scalar on both representations.
\end{proof}

We can now prove Corollary~\ref{coro0:compacto}, as an  application of  Theorem~\ref{thm0:d_lambda_cpto}.

\begin{proof}[Proof of Corollary~\ref{coro0:compacto}]
(i) For each $p$, Theorem~\ref{thm0:d_lambda_cpto} yields that the multiplicities of the eigenvalues of the Hodge-Laplace operator on $p$-forms determine the multiplicity $n_\Gamma(\pi)$ of every $\pi\in\widehat G_{\tau_p}$, hence $\Gamma_1$ and $\Gamma_2$ are $\tau_p$-equivalent.
The converse follows from Proposition~\ref{prop2:tau-equiv=>tau-iso}.

(ii) Given $\Gamma_1$ and $\Gamma_2$, assume that $\Gamma_1\ba X$ and $\Gamma_2\ba X$ are $p-1$-isospectral and $p+1$-isospectral. Then, by (i), $\G_1$ and $\G_2$ are $\tau_{p-1}$ and $\tau_{p+1}$-equivalent.

For $0<p<n$, since $\widehat G_{\tau_p}\subset \widehat G_{\tau_{p-1}}\cup \widehat G_{\tau_{p+1}}$ by Proposition~\ref{prop3:G_tau_p}, it follows that $n_{\Gamma_1}(\pi)=n_{\Gamma_2}(\pi)$ for every $\pi\in \widehat G_{\tau_p}$ hence $\G_1$ and $\G_2$ are $\tau_p$-equivalent, and as a consequence $\Gamma_1\ba X$ and $\Gamma_2\ba X$ are $p$-isospectral.

For $p=0$ (resp.\ $p=n$) we have $\widehat G_{\tau_0}\subset \widehat G_{\tau_{1}}\cup\{\pi_{0,\delta}\}$ (resp.\ $\widehat G_{\tau_n}\subset \widehat G_{\tau_{n-1}}\cup\{\pi_{0,-\delta}\}$), hence $\G_1$ and $\G_2$ are $\tau_0$-equivalent (resp.\ $\tau_n$-equivalent) since $n_{\Gamma_i}(\pi_{0,\delta}) = \beta_0(\Gamma_i\ba S^n)=1$ (resp.\ $n_{\Gamma_i}(\pi_{0,-\delta}) = \beta_n(\Gamma_i\ba S^n)=1$). 
Hence $\Gamma_1\ba X$ and $\Gamma_2\ba X$ are $0$-isospectral (resp.\ $n$-isospectral).
\end{proof}

\begin{rem}
The $p$-spectrum of spherical space forms has been investigated by many authors. For instance, in \cite{IkTa}, Ikeda and Taniguchi studied the $p$-spectrum of homogeneous spaces $G/K$  from the point of view of representation theory, determining the eigenvalues and the eigenspaces in the case of $S^n$ and $\C P^n$.
Later, Ikeda~\cite{Ik}, for every  $0\leq p<\frac{n-1}2$, found lens spaces that are $q$-isospectral for every $0\leq q\leq p$ but are not $p+1$-isospectral.
In  \cite{Pe2}, Pesce considered the notion of $\tau$-equivalent discrete subgroups and showed that $\tau$-isospectral spherical space forms give $\tau$-equivalent groups, in the case when the real-eigenspaces in $L^2(S^n;\tau)$ are irreducible. Finally we mention \cite{GM}, where Gornet and McGowan gave a rich family of examples of lens spaces that are $\tau_p$-equivalent for some values of $p$ only.
\end{rem}

As mentioned in the Introduction, Corollary~\ref{coro0:compacto}~(i) can be extended to  $\tau = \tau_{\mu, \kappa} \in\widehat K$, for more general choices of the highest weight $\mu$.

\begin{prop}\label{prop3:mu=2}
Let $\Gamma_1$, $\Gamma_2$ be finite subgroups of $G=\Ot(n+1)$   acting freely on $S^{n}$.
Let $\mu=\sum_{i=1}^{m-1} b_i\,\varepsilon_i\in\mathcal P(\SO(2m))$ be such that
$$2=b_1\geq b_2\geq\dots \geq b_{m-1}\geq 0 $$
and let $\kappa\in\{\pm1\}$.
If $\Gamma_1\backslash S^n$ and $\Gamma_2\backslash S^n$ are $\tau_{\mu,\kappa}$-isospectral, then  $\Gamma_1$ and $\Gamma_2$ are $\tau_{\mu,\kappa}$-equivalent.
\end{prop}
\begin{proof}
As we noted in the proof of Corollary~\ref{coro0:compacto}~(i), it is sufficient to show that different representations of $\widehat G_{\tau_{\mu,\kappa}}$ have different Casimir eigenvalues.
The proof will be divided into two cases:

\begin{enumerate}
\item[(a)] $\mu_p:=2\varepsilon_1+\dots+2\varepsilon_p$ for some $1\leq p\leq m-1$,
\item[(b)] $\mu_{p,q}:=2\varepsilon_1+\dots+2\varepsilon_p+\varepsilon_{p+1}+\dots+\varepsilon_q$ for some $1\leq p<q\leq m-1$.
\end{enumerate}

\smallskip

\noindent {Case (a).}
By the branching law (see for example \cite[Prop.~I.5]{Pe1}) we have that $[\tau_{\mu,\kappa}: \pi_{\Lambda,\delta}|_K]>0$ if and only if
$$
\Lambda = \Lambda_{(k,a),p} := k\varepsilon_1+2\varepsilon_2+\dots+2\varepsilon_p+a\varepsilon_{p+1},
$$
where $k\geq2$, $0\leq a\leq 2$ and  $\delta\in\{0,\pm1\}$ has a specific value.
Hence, the highest weights involved in $\widehat G_{\tau_{\mu_p,\kappa}}$ have the form $\Lambda_{(k,a),p}$ with $0 \leq a \leq 2$, for every $p$.
We have
\begin{multline*}
\lambda(C,\pi_{\Lambda_{(k,a),p},\delta})
    = k(k+2m-2) + \sum_{i=2}^p 2(2+2m-2i) + a(a+2m-2p-2).
\end{multline*}
It remains to prove that $\pi_{\Lambda_{(k,a),p},\delta}=\pi_{\Lambda_{(h,b),p},\delta}$ if and only if $(k,a)=(h,b)$.
Suppose $0\le a<b \le 2$. Then
\begin{align*}
k(k+2m-2) + a(a+2m-2p-2)
    &= h(h+2m-2) + b(b+2m-2p-2)\\
(k-h)(k+h+2m-2)
    &= (b-a)\big(b+a+2m-2p-2\big).
\end{align*}
If $k>h$, $b>a$, since $k+h+2m-2>b+a+2m-2p-2$ then $0<k-h<b-a$. Hence $b-a=2$ and $k-h=1$, thus we have a contradiction since the left-hand side is odd and the right-hand side is even. Therefore, necessarily, $k=h$ and $b=a$, as asserted.

\smallskip

\noindent {Case (b).}
The proof is very similar to the previous one, so we only give the main ingredients.
The highest weights involved in $\widehat G_{\tau_{\mu_{p,q},\kappa}}$ have the form
$$
\Lambda_{(k,a_1,a_2),p} := k\varepsilon_1+2\varepsilon_2+\dots+2\varepsilon_p+a_1\varepsilon_{p+1} +\varepsilon_{p+2}+\dots+\varepsilon_{q}+a_2\varepsilon_{q+1},
$$
where $k\geq2$ and $0\leq a_1\leq 1\leq a_2\leq 2$. In this case,
\begin{multline*}
\lambda(C,\pi_{\Lambda_{(k,a_1,a_2),p},\delta})
    = k(k+2m-2) + \sum_{i=2}^p 2(2+2m-2i)  \\
    + a_1(a_1+2m-2p-2)  +\sum_{i=p+2}^q (1+2m-2i) + a_2(a_2+2m-2q-2).
\end{multline*}
Suppose $\lambda(C,\pi_{\Lambda_{(k,a_1,a_2),p},\delta}) = \lambda(C,\pi_{\Lambda_{(h,b_1,b_2),p},\delta})$ with $a_2<b_2$, i.e.\ $a_2=0$ and $b_2=1$.
One can check that
\begin{multline*}
(k-h)(k+h+2m-2)
    = (b_1-a_1)(b_1+a_1+2m-2p-2) + 1+2m-2q-2.
\end{multline*}
In case $b_1=a_1$ we arrive at a contradiction as above. If $a_1=1$ and $b_1=2$, then the right-hand side is equal to $4m-2(p+q)$, hence $k-h$ is an even positive integer, thus the right-hand side is greater than the left-hand side.
If $a_1=2$ and $b_1=1$, the right-hand side is equal to $-2(q-p+1)$ and again we arrive at a contradiction as before.
\end{proof}

\begin{rem}
Note that Proposition~\ref{prop3:mu=2} follows again from the fact that, for any $\lambda\in\R$, in formula \eqref{multld} at most one irreducible representation in $\widehat G_{\tau_{\mu,\kappa}}$ gives a contribution.
This is not true generically  for $\tau\in\widehat K$.
For instance, for $\tau = \tau_{\mu,\kappa}$ with $\mu=3\varepsilon_1$ and $\kappa=\pm1$, set
$\Lambda=2m\varepsilon_1$ and $\Lambda'=(2m-1)\varepsilon_1+3\varepsilon_2$, thus $\pi_{\Lambda,\delta}, \pi_{\Lambda',\delta}\in \widehat G_\tau$ for a single value of $\delta$ and we have
\begin{align*}
\lambda(C,\pi_{\Lambda,\delta})
    &= \langle\Lambda, \Lambda+2\rho\rangle = 2m\big(2m+2(m-1)\big) = 2n(n+1),\\
\lambda(C,\pi_{\Lambda',\delta})
    &= (2m-1)\big(2m-1+2(m-1)\big) + 3\big(3+2(m-2)\big)  \\
    &= n(n+n-1) + 3n = 2n(n+1).
\end{align*}
Therefore the eigenspace of $\Delta_{\tau_{\mu,\kappa},\Gamma}$ for the eigenvalue $\lambda = 2n(n+1)$ is equal to $\pi_{\Lambda,\delta}\oplus \pi_{\Lambda',\delta}$, which is not irreducible.
\end{rem}

\begin{rem}\label{rmk3:Hodge}
Let $\Omega_p(M)$ denote the space of differential forms of degree $p$ on a Riemannian compact manifolds $M$.
By the Hodge decomposition at degree $p$
\begin{equation}\label{eq3:hodge}
\Omega_p(M) = \mathcal H_p(M)\oplus \Omega_{p}'(M) \oplus \Omega_{p}''(M),
\end{equation}
where $\mathcal H_p(M)$ denotes the $p$-harmonic forms and  $\Omega_{p}'(M)$ and $\Omega_{p}''(M)$ denote the subspace of  \emph{closed} ($d\Omega_{p-1}(M)$) and \emph{coclosed forms} ($d^*\Omega_{p+1}(M)$)  respectively.
A subscript $\lambda\in\R$ in these sets will denote the restriction to the eigenspace associated to the eigenvalue $\lambda$.
Clearly $\Omega_p(M)_0= \mathcal H_p(M)$ and $\Omega_p(M)_\lambda = \Omega_{p}'(M)_\lambda \oplus \Omega_{p}''(M)_\lambda$ for any $\lambda\neq 0$.

In this case, Theorem~\ref{thm0:d_lambda_cpto} ensures that the sets $\Omega_{p}'(M)_\lambda$ and $\Omega_{p}''(M)_\lambda$ cannot  both be nonempty.
Moreover, the $p$-eigenspace associated to $\lambda\in \mathcal E_p$ (resp.\ $\mathcal E_{p+1}$) is contained in $\Omega_{p}'(M)_\lambda$ (resp.\ $\Omega_{p}''(M)_\lambda$).

Gornet and McGowan introduced  the notion of \emph{half-isospectrality} (see \cite[Rmk.~4.5]{GM})  meaning isospectrality with respect to $\Delta_{\tau_p,\Gamma}$ restricted to closed or coclosed $p$-forms.
They also showed several examples of half-isospectral lens spaces.
In a way similar to Corollary~\ref{coro0:compacto}~(i), we can give an equivalent formulation in terms of representations as follows:
\begin{quote}
  $\Gamma_1\ba S^n$ and $\Gamma_2\ba S^n$ are isospectral on closed (resp.\ coclosed) $p$-forms if and only if $n_{\Gamma_1}(\pi_{\Lambda_{k,p},\delta}) = n_{\Gamma_2}(\pi_{\Lambda_{k,p},\delta})$ (resp.\ $n_{\Gamma_1}(\pi_{\Lambda_{k,p+1},\delta}) = n_{\Gamma_2}(\pi_{\Lambda_{k,p+1},\delta})$) for every $k\in\N$.
\end{quote}
Now, this fact, together with Proposition~\ref{prop3:G_tau_p}, ensure that two spherical space forms are $p+1$-isospectral on closed forms if and only if they are $p$-isospectral on coclosed forms.
In particular, the examples of $p$-isospectral and not $p+1$-isospectral Lens spaces given in \cite{Ik} and \cite{GM}, are examples of manifolds $p+1$-isospectral on closed forms but not on coclosed forms.
\end{rem}

\section{Flat case}\label{sec:flat}

We now consider the flat case $X=\R^n$.
Then
\begin{equation}
  G=\Ot(n)\ltimes\R^n \simeq \Iso(\R^n),
\end{equation}
and $K=\Ot(n)$.
Let $\Gamma$ be a discrete cocompact  subgroup of $G$ acting freely on $\R^n$, i.e.\ a \emph{Bieberbach group}.
Any element $\gamma\in\Gamma\subset G$ decomposes uniquely as $\gamma = B L_b$, with $B\in K$ and $b\in\R^n$.
The matrix $B$ is called the rotational part of $\gamma$ and $L_b$ is called the translational part.
The subgroup $L_\Lambda$ of pure translations in $\Gamma$ is called the {\it translation lattice} of $\Gamma$ and $F:=\Ld\ba\G$ is the {\it point group} (or the {\it holonomy group}) of $\G$.

We need a description of the unitary dual of $G$.
We will use Mackey's method (see \cite[\S 5.4]{Wa}).
We identify $\widehat{\R}^n$ with $\R^n$ via the correspondence $\alpha \rightarrow \xi_\alpha(\,.\,) = e^{2\pi i \langle \alpha,\, .\, \rangle}$ for $\alpha \in \R^n$.
The group $G$ acts on $\widehat{\R}^n$ by $(g\cdot \xi_\alpha)(b)=\xi_\alpha(g^{-1}b)$.
For $\alpha\in\R^n$ we consider $K_\alpha=\{k\in K:k\cdot \xi_\alpha=\xi_\alpha\}$, the stabilizer of $\xi_\alpha$ in $K$.

For $\alpha\in\R^n$ and $(\sigma,V_\sigma)\in \widehat K_\alpha$, we consider the induced representation of $G$ given by
\begin{equation}\label{eq4:rep_E(n)}
(\pi_{\sigma,\alpha},W_{\sigma,\alpha}):=\operatorname{Ind}_{K_\alpha\ltimes\R^n}^{K\ltimes\R^n} (\sigma\otimes \xi_\alpha).
\end{equation}
Here, the space $W_{\sigma,\alpha}$ is the completion of the space
$$
C_{\sigma,\alpha}=\{f:G\to V_\sigma \text{ cont. }: f((k,b)g)=\sigma(k)\xi_\alpha(b)f(g),\;\;\forall k\in K_\alpha,\,b\in\R^n\}
$$
with respect to a canonical inner product. The action of $G$ on $W_{\sigma,\alpha}$ is by right translations.
Since $(\sigma\otimes \xi_\alpha,V_\sigma)$ is unitary, $(\pi_{\sigma,\alpha},W_{\sigma,\alpha})$ is a unitary representation of $G$.
It is well-known that $\pi_{\sigma,\alpha}$ is irreducible and, furthermore, every unitary representation of $G$ is unitarily equivalent to one of this form.

Note that if $\alpha=0$, then $K_\alpha=K=\Ot(n)$.
Furthermore, for $(\tau,V)\in \widehat K$, we have $\widetilde\tau:=\pi_{\tau,0}\simeq \tau\otimes\Id$, i.e.\ $\widetilde\tau(v)=\tau(v)$ for $v\in V$, therefore $(\widetilde\tau,V)\in\widehat G$ is finite dimensional.

On the other hand, if $\alpha\neq0$ and $\sigma\in\widehat K_\alpha$, then $\pi_{\sigma,\alpha}\simeq \pi_{\sigma,re_n}$ where $r=\|\alpha\|$.
We shall abbreviate $\pi_{\sigma,re_n}$ by writing $\pi_{\sigma,r}$ for $r\geq0$.
In this case, $K_\alpha= \left[\begin{smallmatrix}\Ot(n-1)\\&1\end{smallmatrix}\right] \simeq\Ot(n-1)$, when $r>0$.

Summing up, a full set of representatives of $\widehat G$ is given by
\begin{eqnarray}\label{eq4:dual_E(n)}
\widehat G = \{\pi_{\sigma,r}:\sigma\in\widehat{\Ot(n-1)},\; r>0\} \cup \{\widetilde\tau:\tau\in\widehat{\Ot(n)}\}
\end{eqnarray}

Now we determine $\widehat G_{\tau_p}$,
that is, the representations in $\widehat G$ such that its restriction to $\Ot(n)$ contains the $p$-exterior representations $\tau_p$ of $\Ot(n)$.
Recall that $\sigma_p$ denotes the complexified $p$-exterior representations of $\Ot(n-1)$.

\begin{lemma}\label{lem4:G_tau_p}
We have
\begin{align*}
\widehat G_{\tau_p} &=
    \{\pi_{\sigma_p,r},\;\pi_{\sigma_{p-1},r}: r>0 \}
    \cup   \{\widetilde \tau_p\}
\end{align*}
for all $p$.
Moreover $[\tau_p:\pi|_K]=1$ for every $\pi\in\widehat G_{\tau_p}$.
\end{lemma}
\begin{proof}
 Let $\pi_{\sigma,r}\in\widehat G$ with $\sigma\in\widehat {\Ot(n-1)}$ and $r>0$.
Since $\pi_{\sigma,r}|_K=\operatorname{Ind}_{K_\alpha}^K(\sigma)$ and $[\tau_p : \operatorname{Ind}_{K_\alpha}^K(\sigma)]=[\sigma:\tau_p|_{K_\alpha}]$ by Frobenius reciprocity, we have that $[\tau_p:\pi_{\sigma,r}|_K]>0$ if and only if $\sigma=\sigma_p,\sigma_{p-1}$ by Proposition~\ref{prop2:tau_p|_O(n-1)}.

Now if $\widetilde \tau\in \widehat G$ with $\tau\in \widehat K$, then $\widetilde\tau|_{K}=\tau$, it follows that $[\tau_p,\widetilde\tau|_K]>0$ if and only if $\tau=\tau_p$.
\end{proof}

If $e_1, \ldots, e_n$ is the canonical basis of $\R^n$, the operator $C=\sum_{i=1}^n e_i^2 \in  U(\mathfrak g)$  descends to the Hodge-Laplace operator $\Delta_{\tau_p,\Gamma}$ on $p$-forms of $\G\ba\R^n \simeq \G\ba \Iso(\R^n)/\Ot(n)$ (see~Subsection~\ref{subsec:bundles}).
The following lemma tells us how $\Delta_{\tau_p,\Gamma}$ operates on any $\pi\in\widehat G$.

\begin{lemma}\label{lem4:Casimir_on_pi}
The element $C\in U(\mathfrak g)$ acts on $\pi\in\widehat G$ by multiplication by a scalar $\lambda(C,\pi)$ given as follows:
\begin{equation*}\label{eq4:Casimir_on_pi}
\lambda(C,\pi) =
\begin{cases}
  0&\;\text{for }\pi=\widetilde\tau,\\
  -4\pi^2\|\alpha\|^2 &\;\text{for }\pi=\pi_{\sigma,\alpha},\;\alpha\neq0.
\end{cases}
\end{equation*}
\end{lemma}
\begin{proof}
In the first case   $\widetilde\tau(k,v)=\tau(k)$, for any $k\in K$, $v\in\R^n$.
If $X\in\R^n$,
$$
\widetilde\tau(X)(k,v)=\left.\frac{d}{dt}\right|_{0} \widetilde\tau (k,v+tX)
=\left.\frac{d}{dt}\right|_{0}\tau(k)=0.
$$
If $\pi=\pi_{\sigma,\alpha}$ with $\alpha\neq0$ and $f\in C_{\sigma,\alpha}$, then
\begin{align*}
\pi_{\sigma,\alpha}(X)f(k,v)
    &=\left.\frac{d}{dt}\right|_{0} f(k,v+tX)
    =\left.\frac{d}{dt}\right|_{0} f((1,tk\cdot X)\cdot(k,v))\\
    &=\left.\frac{d}{dt}\right|_{0} e^{2\pi i t \langle \alpha,(k\cdot X)\rangle} f(k,v)
    =2\pi i\langle k^{-1}\cdot\alpha,X\rangle f(k,v).
\end{align*}
Thus $\displaystyle
\pi_{\sigma,\alpha}(C)f (k,v)= -4\pi^2\sum_{i=1}^n \langle k^{-1}\alpha,e_i\rangle^2 f(k,v) =-4\pi^2\|\alpha\|^2 f(k,v).$
\end{proof}

Now we are in a condition to prove the results in the Introduction in the flat case.

\begin{proof}[Proof of Theorem~\ref{thm0:d_lambda_flat}]
By Proposition~\ref{prop2:multip_lambda}, given an eigenvalue $\lambda\in\R$ of the Hodge-Laplace operator on $p$-forms $\Delta_{\tau_p,\Gamma}$ on $\Gamma\ba\R^n$,  the multiplicity $d_\lambda(\tau_p,\Gamma)$ is given by $\sum\, n_{\Gamma}(\pi) \; [\tau_p:\pi|_K]$, where the sum is over every $\pi\in\widehat G_{\tau_p}$ such that $-\lambda(C,\pi)=\lambda$.
Now, by using Lemma~\ref{lem4:G_tau_p} and Lemma~\ref{lem4:Casimir_on_pi} we obtain that
\begin{equation}\label{eq4:mult_p}
d_\lambda(\tau_p,\Gamma) =
\begin{cases}
n_\Gamma(\widetilde\tau_p) &\;\text{if }\lambda=0,\\
n_\Gamma(\pi_{\sigma_p,\sqrt\lambda/2\pi})+ n_\Gamma(\pi_{\sigma_{p-1},\sqrt\lambda/2\pi})
    &\;\text{if }\lambda>0,
\end{cases}
\end{equation}
and thus Theorem~\ref{thm0:d_lambda_flat} follows.
\end{proof}

We will use the following Lemma to prove Theorem~\ref{thm0:A} in the flat case and other consequences in Corollary~\ref{coro4:corolarios}.

\begin{lemma}\label{lem4:p-1-equiv,p-iso0>p-equiv}
If $\,\Gamma_1$ and $\Gamma_2$ are  $\tau_{p-1}$-equivalent (or $\tau_{p+1}$-equivalent) and $\Gamma_1\ba\R^n$ and $\Gamma_2\ba\R^n$ are $p$-isospectral, then $\Gamma_1$ and $\Gamma_2$ are $\tau_{p}$-equivalent.
\end{lemma}

\begin{proof}
Since $\Gamma_1$ and $\Gamma_2$ are $\tau_{p-1}$-equivalent we have that $n_{\Gamma_1}(\pi_{\sigma_{p-1},r}) = n_{\Gamma_2}(\pi_{\sigma_{p-1},r})$ for every $r>0$ by Proposition~\ref{prop3:G_tau_p}.
On the other hand, since $\Gamma_1\ba\R^n$ and $\Gamma_2\ba\R^n$ are $p$-isospectral we have that $n_{\Gamma_1}(\widetilde\tau_{p})=n_{\Gamma_2}(\widetilde\tau_{p})$ and
$$
n_{\Gamma_1}(\pi_{\sigma_{p},r}) + n_{\Gamma_1}(\pi_{\sigma_{p-1},r})= n_{\Gamma_2}(\pi_{\sigma_{p},r})+ n_{\Gamma_2}(\pi_{\sigma_{p-1},r})
$$
for any $r>0$, by \eqref{eq4:mult_p}.
These three facts taken together, clearly imply $\tau_{p}$-equivalence.
The assertion assuming $\tau_{p+1}$-equivalence follows in a similar way.
\end{proof}

\begin{proof} [Proof of Theorem~\ref{thm0:A} (flat case)]\label{proof4:thm0}
The fact that $\tau_p$-equivalence implies $p$-iso\-spectra\-lity is clear in light of Proposition~\ref{prop2:tau-equiv=>tau-iso}.
For the converse assertion, we proceed by induction.
Lemma~\ref{lem4:p-1-equiv,p-iso0>p-equiv} for $p=0$ says that $0$-isospectrality implies $\tau_0$-equivalence.
Now assume that the manifolds are $q$-isospectral for every $0\leq q\leq p$, thus we have that the groups are $\tau_{q}$-equivalent for every $0\leq q\leq p-1$ by the induction hypothesis.
In particular we have $\tau_{p-1}$-equivalence and $p$-isospectrality, hence Lemma~\ref{lem4:p-1-equiv,p-iso0>p-equiv} implies $\tau_p$-equivalence, which completes the proof.
\end{proof}

\begin{rem}
One can also prove the above result for intervals decreasing from $n$, that is:
\emph{$q$-isospectrality for every $p\leq q\leq n$ is equivalent to $\tau_q$-equivalence for every $p\leq q\leq n$.}
\end{rem}

We can also obtain from Theorem~\ref{thm0:d_lambda_flat} several other consequences relating $p$-isospectrality and $\tau_p$-equivalence. Given a compact $n$-manifold $M$,
$\beta_{p}(M)$ denotes the $p^{\textrm{th}}$-Betti number of $M$ and one has that $\beta_{p}(M) = d_0(\tau_p,M)$, the multiplicity of the eigenvalue $0$ of the Hodge-Laplace operator on $p$-forms of $M$.

\begin{prop}\label{prop4:other_consequences}
Let $\Gamma_1$ and $\Gamma_2$ be Bieberbach groups and let $\Gamma_1\ba\R^n$ and $\Gamma_2\ba\R^n$ be the corresponding flat Riemannian manifolds.
Then the following assertions hold.
\smallskip
\begin{enumerate}
\item\label{item4:1-equiv=>0,1-isosp}
    If $\,\G_1$ and $\G_2$ are $\tau_1$-equivalent, then $\G_1\ba\R^n$ and $\G_2\ba\R^n$  are $0$ and $1$-isospectral.
\smallskip

\item\label{item4:n-1-equiv=>n-1,n-isosp}
    If $\,\G_1$ and $\G_2$ are $\tau_{n-1}$-equivalent and $\beta_{n}(\G_1\ba\R^n)=\beta_{n}(\G_2\ba\R^n)$, then  $\G_1\ba\R^n$ and $\G_2\ba\R^n$  are $n$ and $n-1$-isospectral.
\smallskip

\item\label{item4:k,k+2-equiv=>k+1-equiv}
If $\,\G_1$ and $\G_2$ are $\tau_{p-1}$ and $\tau_{p+1}$-equivalent and $\beta_{p}(\G_1\ba\R^n) = \beta_{p}(\G_2\ba\R^n)$, then $\G_1$ and $\G_2$ are also $\tau_{p}$-equivalent, hence $\G_1\ba\R^n$ and $\G_2\ba\R^n$ are $p-1$, $p$ and $p+1$-isospectral.
\end{enumerate}
\end{prop}

\begin{proof}
We will use repeatedly the facts
\begin{align}
\widehat G_{\tau_p}
    &=  \{ \pi_{\sigma_p,r}, \pi_{\sigma_{p-1},r}  : r>0 \}\cup  \{ \widetilde\tau_p\}, \label{hatGtau}\tag{$*$} \\
d_\lambda(\tau_p,\Gamma_i)
    &=
    \begin{cases}
    n_{\Gamma_i}(\widetilde\tau_p) &\;\text{if }\lambda=0,\\
    n_{\Gamma_i}(\pi_{\sigma_p,\sqrt\lambda/2\pi}) + n_{\Gamma_i}(\pi_{\sigma_{p-1},\sqrt\lambda/2\pi})
    &\;\text{if }\lambda>0.
    \end{cases}\label{d_lambda}\tag{$**$}
\end{align}
from Lemma~\ref{lem4:G_tau_p} and Theorem~\ref{thm0:d_lambda_flat}.

We first prove \ref{item4:1-equiv=>0,1-isosp}.
Suppose that $\Gamma_1$ and $\Gamma_2$ are $\tau_1$-equivalent, then $\Gamma_1\ba\R^n$ and $\Gamma_2\ba\R^n$ are $1$-isospectral by Proposition~\ref{prop2:tau-equiv=>tau-iso}.
Since $\pi_{\sigma_0,r}\in\widehat G_{\tau_1}$ for $r> 0$, \eqref{hatGtau} and \eqref{d_lambda} imply that $d_\lambda(\tau_0,\Gamma_1)=d_\lambda(\tau_0,\Gamma_2)$ for every $\lambda>0$, hence $\G_1\ba\R^n$ and $\G_2\ba\R^n$ are also $0$-isospectral, since $d_0(\tau_0,\Gamma_1)=d_0(\tau_0,\Gamma_2)=1$.

Assertion (ii) follows in a similar way by using that $d_0(\tau_n,\Gamma_i)=\beta_{n}(\Gamma_i\ba\R^n)$.

Relative to \ref{item4:k,k+2-equiv=>k+1-equiv} if $\Gamma_1$ and $\Gamma_2$ are $\tau_{p-1}$ and $\tau_{p+1}$ equivalent, then on the one hand, $n_{\Gamma_1}(\pi_{\sigma_{p-1},r})=n_{\Gamma_2}(\pi_{\sigma_{p-1},r})$ for every $r>0$ since $\pi_{\sigma_{p-1},r}\in\widehat G_{\tau_{p-1}}$ and, on the other hand, since $\pi_{\sigma_{p},r}\in\widehat G_{\tau_{p+1}}$, then  $n_{\Gamma_1}(\pi_{\sigma_{p},r})=n_{\Gamma_2}(\pi_{\sigma_{p},r})$ for every $r>0$.
Finally, the equality of the $p^{\textrm{th}}$-Betti numbers implies that $n_{\Gamma_1}(\widetilde\tau_{p} )=n_{\Gamma_2}(\widetilde\tau_{p})$ by \eqref{d_lambda}, thus $\Gamma_1$ and $\Gamma_2$ are $\tau_{p}$-equivalent.
\end{proof}

Note that the condition $\beta_{n}(\G_1\ba\R^n) = \beta_{n}(\G_2\ba\R^n)$ in Proposition~\ref{prop4:other_consequences}~\ref{item4:n-1-equiv=>n-1,n-isosp} is equivalent to $\Gamma_1\ba\R^n$ and $\Gamma_2\ba\R^n$ being both orientable or both non-orientable.
A flat manifold $\Gamma\ba\R^n$ is orientable if and only if $\Gamma\subset\SO(n)\ltimes \R^n$.

The next result  follows immediately from Lemma~\ref{lem4:p-1-equiv,p-iso0>p-equiv} and will be applied in explicit examples.

\begin{coro}\label{coro4:corolarios}
Let $\Gamma_1$ and $\Gamma_2$ be Bieberbach groups.
If $\,\Gamma_1\ba\R^n$ and $\Gamma_2\ba\R^n$ are $p$-isospectral for every $p\in\{1,2,\dots,k\}$ and are not $0$-isospectral, then $\Gamma_1$ and $\Gamma_2$ are not $\tau_p$-equivalent for any $p\in\{0,1,2,\dots,k+1\}$.
Similarly, if $\beta_{n}(\G_1\ba\R^n)=\beta_{n}(\G_2\ba\R^n)$ and $\Gamma_1\ba\R^n, \Gamma_2\ba\R^n$ are $p$-isospectral for every $p\in\{n-k,\dots, n-2, n-1\}$ and are not $n$-isospectral, then $\Gamma_1$ and $\Gamma_2$ are not $\tau_p$-equivalent for any $p\in\{n-k-1,n-k, \dots,n\}$
\end{coro}

\begin{rem}\label{rmk4:Hodge}
We now study the Hodge decomposition of a compact flat manifold as in Remark~\ref{rmk3:Hodge}.
In this case, Theorem~\ref{thm0:d_lambda_flat} implies that $\mathcal H_p(M)_0$ is the $0$-eigenspace associated to $\widetilde \tau_p$ and for $\lambda\neq0$, again we have  $\Omega_p(M)_\lambda=\Omega_{p}'(M)_\lambda \oplus \Omega_{p}''(M)_\lambda$, where both can be nonempty at the same time.

Unlike the notion of $p$-isospectrality, we have an equivalent definition of compact flat manifolds $p$-isospectral on closed forms (resp.\ coclosed forms) in terms of representations. Namely from Lemma~\ref{lem4:G_tau_p} one can see that
\begin{quote}
  $\Gamma_1\ba \R^n$ and $\Gamma_2\ba \R^n$ are isospectral on closed (resp.\ coclosed) $p$-forms if and only if $n_{\Gamma_1}(\pi_{\sigma_{p},r}) = n_{\Gamma_2}(\pi_{\sigma_{p},r})$ (resp.\ $n_{\Gamma_1}(\pi_{\sigma_{p-1},r}) = n_{\Gamma_2}(\pi_{\sigma_{p-1},r})$) for every $r>0$.
\end{quote}

In this way we can find many examples of compact flat manifolds that are \emph{half-isospectral} but not isospectral.
For instance, if  they are $0$-isospectral and not $1$-isospectral, then they are $1$-isospectral on closed forms but not on coclosed forms.
Examples of this type can be found in \cite[Examples 5.1, 5.5, 5.9]{MRp}.
\end{rem}

In the rest of this section we give several examples of compact flat manifolds satisfying some $p$-isospectralities or $\tau_p$-equivalences for some values of $p$ only. We   denote by $\{e_1,\dots,e_n\}$ the canonical basis of $\R^n$.

We recall from \cite[Thm.~3.1]{MRp} that the multiplicity of the eigenvalue $4\pi^2 \mu$ of $\Delta_{\tau_p,\Gamma}$
is given by
\begin{equation}\label{eq4:mult_mu}
d_{4\pi^2\mu}(\tau_p,\Gamma)
= |F|^{-1} \sum_{\gamma=BL_b\in \Lambda \backslash \Gamma}\tr_p(B) \, e_{\mu,\gamma}(\Gamma),
\end{equation}
where $e_{\mu,\gamma}(\Gamma):= \sum_{v\in {\Lambda^*_\mu}:Bv=v} e^{-2\pi i v\cdot b}$, $\Lambda^*_\mu:=\{v \in \Lambda^* : \| v\|^2=\mu\}$ ($\Lambda^*$  the dual lattice of $\Lambda$) and $\tr_p (B):=\tr(\tau_p(B))$.
If $p=0$, \eqref{eq4:mult_mu}  reads
\begin{equation}\label{eq4:mult_mu0}
d_{4\pi^2\mu}(\tau_0,\Gamma)
= |F|^{-1}  \sum_{\gamma=BL_b\in \Lambda \backslash \Gamma}  \; \sum_{v\in {\Lambda^*_\mu}:Bv=v} e^{-2\pi i v\cdot b}.
\end{equation}

\begin{exam}\label{exam4:Klein_bottles}
We first show a pair of non isometric Klein bottles that are $1$-isospectral but not $0$-isospectral, hence the corresponding Bieberbach groups  cannot be $\tau_1$-equivalent by Proposition~\ref{prop4:other_consequences}~\ref{item4:1-equiv=>0,1-isosp}.

Let $\Gamma=\langle \gamma,L_\Lambda\rangle$ and $\Gamma'=\langle \gamma', L_\Lambda\rangle$, where $\Lambda=\Z e_1\oplus \Z c e_2$ with $c>1$ and in \emph{column notation}
\begin{equation}\label{column}
\begin{array}{|r@{}l|}
\multicolumn{2}{c}{\g}\\
\hline
1&_{\f12} \\
-1& \\
\hline
\end{array}
\qquad\text{and}\qquad
\begin{array}{|r@{}l|}
\multicolumn{2}{c}{\g'}\\
\hline
-1& \\
1&_{\f{1}2} \\
\hline
\end{array}\,.
\end{equation}
That means that $\gamma=B L_b$ and $\gamma'=B' L_{b'}$ where $B=\left[\begin{smallmatrix}1\\&-1\end{smallmatrix}\right]$, $B'=\left[\begin{smallmatrix}-1\\&1\end{smallmatrix}\right]$, $b=\frac12 e_1$ and $b'= \frac 12 ce_2$,
i.e.\ the  column in \eqref{column} gives the rotation part of $\g$, $\g'$ and the subindices indicate their translation vectors.

The manifolds $\Gamma\ba\R^n$ and $\Gamma'\ba\R^n$ are $1$-isospectral in light of \eqref{eq4:mult_mu} since $\tr_1(B)=\tr_1(B')=0$.
However, they are not $0$-isospectral since, by using \eqref{eq4:mult_mu0}, one can see that the smallest eigenvalue for $\Gamma\ba\R^n$,  $\lambda=4\pi^2c^{-2}$, has multiplicity $2$ while $\lambda$ is not an eigenvalue for $\Gamma'\ba\R^n$.

The Klein bottles just defined are homeomorphic. However, it is not hard to give a pair of \emph{non homeomorphic} compact flat $4$-manifolds that are $1$-isospectral but not $0$-isospectral.
We define $\Gamma=\langle \gamma,L_{\Z^4}\rangle$ and $\Gamma'=\langle \gamma', L_{\Z^4}\rangle$ where, in column notation,
\begin{equation*}
\begin{array}{|r@{}l|}
\multicolumn{2}{c} {\g} \\
\hline
1&_{\f12} \\
1& \\
-1&\\
-1& \\
\hline
\end{array}
\qquad\text{and}\qquad
\begin{array}{|r@{}l|}
\multicolumn{2}{c} {\g'} \\
\hline
1&_{\f12} \\[.25cm]
\multicolumn{2}{|c|}{J}\\[.25cm]
-1& \\
\hline
\end{array}.
\end{equation*}
Here $J=\left[\begin{smallmatrix}0&1\\1&0\end{smallmatrix}\right]$ and   $\gamma'=B'L_{b'}$
with $B'=\op{diag}(1,J,-1)\in \GL(4,\R)$ and $b'=(1/2,0,0,0)^t\in\R^4$.

These manifolds are $1$-isospectral because, again, $\tr_1(B)=\tr_1(B')=0$. They are not $0$-isospectral.
Indeed, it follows easily from \eqref{eq4:mult_mu0} that  the smallest nonzero eigenvalue is $4\pi^2$ for both manifolds, but it has multiplicity $\frac12 (8+0)=2$ for the first one and $\frac12 (8-2)=3
$ for the second one.

One can show, by using the theory of Bieberbach groups, that these manifolds cannot be homeomorphic since the holonomy representations are not semiequivalent.
\end{exam}

\begin{exam}\label{exam4:n=4_p_odd_iso}
We now give a pair of $4$-dimensional compact flat manifolds that are $p$-isospectral for $p=1,3$ and they are not $p$-isospectral for $p=0,2,4$.
The corresponding Bieberbach groups cannot be $\tau_p$-equivalent for any $p$, by Proposition~\ref{prop2:tau-equiv=>tau-iso}, for $p$ even and by Proposition~\ref{prop4:other_consequences}~\ref{item4:1-equiv=>0,1-isosp}-\ref{item4:n-1-equiv=>n-1,n-isosp}, for $p$ odd.

The manifolds mentioned are called $M_{24}$ and $M_{25}$ in the notation in \cite[Example 4.8]{CMR}, and can be described as $\Gamma=\langle \gamma_1,\gamma_2,L_{\Z^4}\rangle$ and $\Gamma'=\langle \gamma_1',\gamma_2', L_{\Z^4}\rangle$ where
\begin{equation*}
\begin{array}{|r@{}l|r@{}l|}
\multicolumn{2}{c} {\gamma_1} & \multicolumn{2}{c} {\gamma_2}\\ \hline
-1&       & 1& \\
-1&       &-1&_{\f12} \\
 1&       &-1&  \\
 1&_{\f12}& 1&_{\f12}  \\
\hline
\end{array}
\qquad\qquad
\begin{array}{|r@{}l|r@{}l|}
\multicolumn{2}{c} {\gamma_1'} & \multicolumn{2}{c} {\gamma_2'}\\ \hline
-1&       & 1&_{\f12} \\
-1&       &-1&_{\f12} \\
 1&       &-1&  \\
 1&_{\f12}& 1&  \\
\hline
\end{array}
\end{equation*}
The manifolds $\Gamma\ba\R^n$ and $\Gamma'\ba\R^n$ are non homeomorphic since they have different homology over $\Z_2$.
Indeed, one has that $\beta_1^{\Z_2}(M_{24})=4 \ne \beta_1^{\Z_2}(M_{25})=3$ and $\beta_2^{\Z_2}(M_{24})=6\ne \beta_2^{\Z_2}(M_{25})=4$.
\end{exam}

\begin{exam}\label{exam4:chino}
This is a curious example of two $8$-dimensional flat manifolds which are $p$-isospectral for every $p\in\{1,2,3,5,6,7\}$ but not for $p\in\{0,4,8\}$.
According to Corollary~\ref{coro4:corolarios}, the corresponding Bieberbach groups cannot be $\tau_p$-equivalent for any $p$.

We define $\Gamma=\langle \gamma, L_{\Z^8}\rangle$ and $\Gamma'=\langle \gamma', L_{\Z^8}\rangle$ where
\begin{equation*}
\begin{array}{|r@{}l|r@{}l|r@{}l|}
\multicolumn{2}{c}{\gamma} & \multicolumn{2}{c}{\gamma^2} & \multicolumn{2}{c}{\gamma^3} \\
\hline
\ovr {}{\widetilde J}& & -I &  & -\widetilde J& \\[.25cm]
\widetilde J & & -I & & -\widetilde J & \\[.25cm]
1&_{\f14} & 1&_{\f12} & 1&_{\f34} \\
1& & 1& & 1& \\
-1&  & 1&  & -1&  \\
-1&  & 1&  & -1&  \\
\hline
\end{array}
\qquad\qquad
\begin{array}{|r@{}l|r@{}l|r@{}l|}
\multicolumn{2}{c}{\gamma'} & \multicolumn{2}{c}{{\gamma'}^2} & \multicolumn{2}{c}{{\gamma'}^3} \\
\hline
\ovr {}{\widetilde J}& & -I &  & -\widetilde J& \\[.25cm]
\widetilde J & & -I & & -\widetilde J & \\[.25cm]
1&_{\f14} & 1&_{\f12} & 1&_{\f34} \\
1&_{\f12} & 1& & 1&_{\f12} \\
-1&  & 1&  & -1&  \\
-1&  & 1&  & -1&  \\
\hline
\end{array}.
\end{equation*}
Here $\widetilde J= \left[\begin{smallmatrix}0&1\\ -1&0\end{smallmatrix}\right]$ and $I$ is the $2\times2$ identity matrix.
The elements $\gamma$ and $\gamma'$ have order $4$, thus the manifolds $\Gamma\ba\R^n$ y $\Gamma'\ba\R^n$ have holonomy group isomorphic to $\Z_4$.
We include also the elements $\gamma^2$, $\gamma^3$, ${\gamma'}^2$ and ${\gamma'}^3$ to facilitate the computation of the multiplicities of the eigenvalues.
Note that the only difference between their generators lies only in the sixth coordinate of the translational part, in particular we have $B=B'$ and $\gamma^2={\gamma'}^2$.

We shall compare the spectra of $\Gamma\ba\R^n$ and $\Gamma'\ba\R^n$ by using the formula \eqref{eq4:mult_mu} for the multiplicities of the eigenvalues of the Hodge-Laplace operator on $p$-forms.
The manifolds are $1$-isospectral since $\tr_1(B^k)=0$ for $k=1,2,3$.
One can check that $\tr_2(B)=\tr_2(B^3)=0$ (resp.\ $\tr_3(B)=\tr_3(B^3)=0$), hence the manifolds are $2$-isospectral (resp.\ $3$-isospectral) since the equality in \eqref{eq4:mult_mu} follows from the fact that $\gamma^2={\gamma'}^2$.
The manifolds cannot be $0$-isospectral since the first nonzero eigenvalue $\lambda=4\pi^2$ has different multiplicity in both cases.
Indeed, $d_{\lambda}(\tau_0,\Gamma)=6\neq4 =d_{\lambda}(\tau_0,\Gamma')$.
Since $\det(B)=1$ the manifolds are orientable and then the previous reasoning is valid for $p=5,6,7,8$.
Finally, they cannot be $4$-isospectral since one checks that $\tr_4(B)=\tr_4(B^3)=-2$, $\tr_4(B^2)=6$ and then, by \eqref{eq4:mult_mu}, we obtain that the first nonzero eigenvalue $\lambda=4\pi^2$ has multiplicities $d_\lambda(\tau_0,\Gamma)= 284\neq 288=d_\lambda(\tau_0,\Gamma')$.

These two compact flat manifolds are homeomorphic to each other, but it is not difficult to obtain a similar example with non homeomorphic groups.
Namely we take
\begin{equation*}
\begin{array}{|r@{}l|r@{}l|r@{}l|}
\multicolumn{2}{c}{\gamma} & \multicolumn{2}{c}{\gamma^2} & \multicolumn{2}{c}{\gamma^3} \\
\hline
\ovr {}
{\widetilde J}& & -I &  & -\widetilde J& \\[.25cm]
\widetilde J & & -I & & -\widetilde J & \\[.25cm]
1&_{\f14} & 1&_{\f12} & 1&_{\f34} \\
1&_{\f14} & 1&_{\f12} & 1&_{\f34} \\
-1&  & 1&  & -1&  \\
-1&  & 1&  & -1&  \\
\hline
\end{array}
\qquad\text{and}\qquad
\begin{array}{|r@{}l|r@{}l|r@{}l|}
\multicolumn{2}{c}{\gamma'} & \multicolumn{2}{c}{{\gamma'}^2} & \multicolumn{2}{c}{{\gamma'}^3} \\
\hline
\ovr {}
{\widetilde J}& & -I &  & -\widetilde J& \\[.25cm]
\widetilde J & & -I & & -\widetilde J & \\[.25cm]
J & {\scriptsize  \begin{array}{c}\f12\\[1mm] 0\end{array} }
& I &{\scriptsize  \begin{array}{c}\f12\\[1mm] \f12\end{array} }
& J & {\scriptsize \begin{array}{c}0\\[1mm] \f12\end{array} } \\[.2cm]
1&  & 1&  & 1&  \\
-1&  & 1&  & -1&  \\
\hline
\end{array}.
\end{equation*}
\end{exam}

\section{Negative curvature case}\label{sec:negative}

The goal of this section is to consider the $p$-spectrum of compact hyperbolic manifolds in connection with $\tau_p$-isospectrality.
We set $G= \SO(n,1)$, $K= \Ot(n)$,  $X\simeq \Hy^n$ and $ X_\G \simeq \G\ba\SO(n,1)/K$
thus $X=\Hy^n$ the $n$-dimensional hyperbolic space.
Let $\G\subset\SO(n,1)$ be a discrete cocompact subgroup acting without fixed points on $X$.
We recall that $\SO(n,1)$ is the group of linear transformations on $\R^{n+1}$ preserving the Lorentzian form of signature $(n,1)$ and determinant one.

We will need a description of $\widehat G$.
We will first introduce the principal series representation of $G$.
The group $G$ has an Iwasawa decomposition $G=NAK$, with a corresponding decomposition
$\mathfrak g=\mathfrak k\oplus\mathfrak a\oplus\mathfrak n$ at the Lie algebra level, where $N$ is nilpotent and $A$ is abelian of dimension one.
Let $M$ be the centralizer of $A$ in $K$. One has $M \simeq \Ot(n-1)$.
The Lie subgroup $P=MAN$ of $G$ is a minimal parabolic subgroup of $G$.

If $\nu\in\mathfrak a^*_\C$, then $\xi_\nu(a)=a^\nu$ defines a character of $A$.
We set $\rho_{\mathfrak a}=\f12 (\dim\mathfrak g_\alpha)\, \alpha= \frac{n-1}2 \,\alpha$ where $\alpha$ is the simple root of the pair ($\mathfrak g, \mathfrak a $).
If $(\sigma,V_\sigma)\in\widehat M$ and $\nu\in\mathfrak a^*_\C$, then we let $C_{\sigma,\nu}$  be the space
$$
\left\{ f \textrm{ cont.} : G\rightarrow V_\sigma :
f(mang)=a^{\nu+\rho_{\mathfrak a}}\sigma(m)f(g),\,\forall\, m\in M,a\in A,n\in N \right\}.
$$
If $\langle\,,\,\rangle$ is an $M$-invariant inner product on $V_\sigma$, for $f_1,f_2\in C_{\sigma,\nu}$ set
$$
\langle f_1,f_2\rangle=\int_{M\ba K}\langle f_1(k),f_2(k)\rangle\,dk.
$$
Then $(C_{\sigma,\nu},\langle\,,\,\rangle)$ is a prehilbert space and the Hilbert space completion is denoted by $H_{\sigma,\nu}$.
The action of $G$ by right translations on $C_{\sigma,\nu}$ extends to $H_{\sigma,\nu}$ defining a continuous series of representations of $G$, $(\pi_{\sigma,\nu},H_{\sigma,\nu})$, that is unitary if $\nu\in i\mathfrak a^*$.
It is called the \emph {principal series representations} of $G$.
They are generically irreducible and play a main role in the description of the irreducible representations of $G$.

One usually identifies $\mathfrak a^*_\C$ with $\C$ via the map $\nu\rightarrow\nu(H_0)$, where $H_0 \in \mathfrak a$ satisfies $\alpha(H_0)=1$, in such a way that $\alpha\rightarrow 1$ and $\rho_{\mathfrak a}\rightarrow \frac{n-1}{2}$.
In this way, $\pi_{\sigma,\nu}$ is unitary if $\nu\in i\R$, as mentioned above.

A Hilbert representation $(\pi,H)$ of $G$ is said to be \emph{square integrable} if any $K$-finite matrix coefficient lies in $L^2(G)$.
These representations  were classified by Harish-Chandra and form the so called \emph{discrete series representations of} $G$, denoted $\widehat G_d$.

The determination of the irreducible unitary representations of a general noncompact semisimple Lie group is an open problem, but is known in the particular case of Lie groups of real rank one (see \cite{BB} and also \cite{KS}).
In the case at hand of $G=\SO(n,1)$ one has:

\begin{thm}\label{thm5:unitarydual}
The unitary dual of $G=\SO(n,1)$ consists of
\begin{enumerate}
\item [(i)] the unitary principal series $\pi_{\sigma,\nu}$ for $\nu \in i\R_{\geq0}$, $\sigma \in \widehat M$,
\item [(ii)] the complementary series $\pi_{\sigma,\nu}$ for $0\le \nu < \rho_\sigma$,
\item [(iii)] unitarizable Langlands quotients $J_{\sigma,\rho_\sigma}$,
\item [(iv)] $\widehat G_d$, the discrete series representations of $G$. For $n$ odd one has $\widehat G_d =\emptyset$.
\end{enumerate}
    The number $\rho_\sigma$ in (ii) has the form $\rho-q$ with $q \in \N_0$, $q\le \rho$, where  $q$ depends on the highest weight of $\sigma$.
\end{thm}

The following theorem gives a description of the subset $\widehat G_{\tau_p}$ of $\widehat G$, that is all we need for the purpose of this paper.

\begin{prop}\label{prop5:hatGtau_neg}
Let $\tau_p$ and $\sigma_p$ be the complexified $p$-exterior representations of $K\simeq \Ot(n)$ and $M\simeq\Ot(n-1)$ respectively.
If $0\leq p\leq n$ and $p\neq \frac n2$, then
\begin{align*}
\widehat{G}_{\tau_p}
    &=\{ \pi_{\sigma_p,\nu} : \nu\in i\R_{\geq0}\cup (0, \rho_p) \} \\
    &\quad \cup \{ \pi_{\sigma_{p-1},\nu} : \nu\in i\R_{\geq0} \cup (0,
\rho_{p-1}) \}
    \cup \, \{J_{\sigma_p,\rho_p}, J_{\sigma_{p-1},\rho_{p-1}} \}.
\end{align*}
Here $\rho_p= \rho_{\mathfrak a}-\min(p,n-1-p)$ and $\rho_{\mathfrak a}=\tfrac{n-1}2$.
In particular,
$$
\widehat {G}_{\tau_0}=\widehat {G}_{\bf 1}=\{ \pi_{\bf 1,\nu} :  \nu\in i\R_{\geq0} \cup (0, \rho_{\mathfrak a}) \}   \cup \{\bf 1 \}.
$$

In the case $n=2m$ and $p=m$, one has
\begin{align*}
\widehat {G}_{\tau_m}
		&= \{\pi_{\sigma_{m-1},\nu} :\nu\in i\R_{\geq0}\cup (0, \tfrac 12)\} \\
		&\quad \cup   \{\pi_{\sigma_m,\nu} :\nu\in i\R_{\geq0}\cup (0, \tfrac 12)\} \cup   \{D_{m}^+ \oplus D_{m}^-  \}.
\end{align*}
Here $D_{\frac n2}^+ \oplus D_{\frac n2}^-$ is the sum of the two
discrete series $D_\frac n2^\pm$ of $\SO(n,1)_0$ having lowest $K$-types
$\tau^\pm _{\frac n2}$.
\end{prop}
\begin{proof}
The spherical case, $p=0$ is well-known, so we assume $p>0$.
As mentioned, the unitarizable Langlands quotients $J_{\sigma,\nu}$ occur only at the endpoints of complementary series  $\nu = \rho_\sigma $.

Since $\tau_p|_M=\sigma_p \oplus \sigma_{p-1}$ by Proposition~\ref{prop2:tau_p|_O(n-1)}, Frobenius reciprocity implies that ${\pi_{\sigma, \nu}|}_{K}$ contains $\tau_p$ if and only if  $\sigma = \sigma _p$ or $\sigma = \sigma_{p-1}$.

Now for $n=2m+1$ and $0\le p \le m$ we have complementary series $\pi_{\sigma_p, \nu}$ for $0<\nu < \rho_p =m-p$ (see \cite[Prop.~49]{KS}) and a Langlands quotient $J_{\sigma_p,\rho_p}$ containing $\tau_p$.
For the $M$-type $\sigma_{p-1}$ we have the same description.

We note that in the extreme cases $p=0$ and $p=n$, one gets $J_{\sigma_0,\rho_{\mathfrak a}}=\bf{1}$ and $J_{\sigma_n, \rho_{\mathfrak a}}= \det$.

For $p>m$, $\pi_{\sigma_p,\nu}$ has complementary series for $0<\nu < \rho_p =p-m$ and a Langlands quotient $J_{\sigma_p,\rho_p}$  at the endpoint, with lowest $K$-type $\tau_p$.
 Since $\widehat G_d=\emptyset$, the description of $\widehat G_{\tau_p}$ for $n$ odd is complete.

\smallskip

We now assume $n=2m$.
If $0\le p\le m-1$ we have complementary series $\pi_{\sigma_p, \nu}$ again for $0<\nu< \rho_p = m-\tfrac{1}2-p$ (see \cite[Prop.~50]{KS}) and a Langlands quotient $J_{\sigma_p,\rho_p}$, both containing $\tau_p$, with a similar description for $\sigma_{p-1}$ in place of $\sigma_{p}$.
For $p\ge m+1$, again $\pi_{\sigma_p, \nu}$ has complementary series for $0<\nu < \rho_p =p-(m-\frac12)$ and a Langlands quotient at the endpoint.
Furthermore, $\widehat{G}_{\tau_p}\cap\widehat G_d = \emptyset$ if $p\ne m$, hence the description of $\widehat {G}_{\tau_p}$ is complete in this case.

Finally, if $p=m$, then $\widehat {G}_{\tau_p}\cap \widehat G_d=\{D_m^+\oplus D_m^-\}$ and the unitary representations that contain $\tau_m$ are the unitary principal series and the complementary series $\pi_{\sigma,\nu}$ for $\sigma = \sigma_{m-1},\sigma_m$ and $\nu\in i\R\cup (0,\frac12)$.
Furthermore, at the endpoint $\tfrac 12$, the representations $\pi_{\sigma_{m}, \frac 12}$ and $\pi_{\sigma_{m-1}, \frac 12}$ are reducible and the $K$-type $\tau_m$ is a $K$-type of the irreducible subrepresentation $D_{m}^+ \oplus D_{m}^-$  with multiplicity $1$.
This completes the proof.
\end{proof}

\begin{prop}\label{prop:Casimireigenvalue}
For $\nu \in \C$, the Casimir eigenvalue for the representation $\pi_{\sigma_p,\nu}$ is given by
 \begin{equation}
\ld(C,{\pi_{\sigma_p,\nu}})= -\nu^2+\rho_p^2 =
  -\nu^2 +(\rho_{\mathfrak a}-\min(p, n-1-p))^2.
\end{equation}
In particular $\lambda(C,J_{\sigma_p,\rho_p})=0$ for every $p$.
Furthermore, $\ld(C,{D^\pm_{\tfrac n2}})=0$.
\end{prop}

\begin{proof}
It is well known that the Casimir eigenvalue for the principal series is given by
\begin{equation}\label{eq:Casimir}
\ld(C,\pi_{\sigma,\nu})=-\nu^2 + \rho_{\mathfrak a}^2  -c_\sigma
\end{equation}
where $c_\sigma=\langle\Ld_\sigma+\rho_M,\Ld_\sigma+\rho_M\rangle- \langle\rho_M,\rho_M\rangle$, $\Lambda_\sigma$ is the highest weight of $\sigma$ and
\begin{equation*}
\rho_M=
\begin{cases}
\sum\limits_{j=1}^m  \,(m-j)\varepsilon_j                 &\quad\text{if } n=2m+1, \\
\sum\limits_{j=1}^{m-1}\,(m-j-\tfrac12) \varepsilon_j    &\quad\text{if } n=2m.
\end{cases}
\end{equation*}
Furthermore, for $\sigma=\sigma_p$ we have $\Ld_{\sigma_p} =\sum_{j=1}^{\min(p,n-p)} \varepsilon_j$ (see Example~\ref{exam2:tau_p-SO(n)}).

We assume first that $0\le p \le [\frac n2]=m$.
By a calculation one can see that
\begin{equation*}
c_\sigma =
\begin{cases}
p+ 2\sum\limits_{j=1}^p(m-j)=p +2mp -p(p+1) & \textrm{ if $n$ is odd},\\
p + 2\sum\limits_{j=1}^p(m-\tfrac 12-j)=p +2(m-\tfrac 12)p -p(p+1) &\textrm{ if $n$ is even}.
\end{cases}
\end{equation*}
Thus, in light of \eqref{eq:Casimir},
\begin{equation*}
\ld(C,{\pi_{\sigma_p,\nu}})=
    \begin{cases}
      -\nu^2 +(m-p)^2
        &\quad\text{if $n=2m+1$},\\
      -\nu^2 +\left(m-p-\tfrac 12\right)^2
        &\quad\text{if $n=2m$},
    \end{cases}
\end{equation*}
which establishes the formula.

On the other hand, for $p>[\tfrac n2]$, one has that $\ld(C,{\pi_{\sigma_p,\nu}}) = \ld(C,{\pi_{\sigma_{n-1-p},\nu}})$ and finally, for $n=2m$, $\ld(C,{D^\pm_m})= \ld(C,\pi_{\sigma_m,\frac 12}) = 0$, as asserted.
\end{proof}

After all this preparation, we can prove the results in the Introduction  for
negatively curved  manifolds.

\begin{proof}[Proofs of Theorem~\ref{thm0:d_lambda_neg}]
For each  $\ld$,  set $\widehat {G}_{\tau_p,\ld} = \{\pi \in {G}_{\tau_p} : \ld(C,\pi)=\ld \}$.
If $p=0$, then the representations in $\widehat {G}_{{\bf 1},\ld}$ for any fixed $\ld$ have the form $\,\pi_{{\bf 1},\nu}\,$ with $\,\nu\in i\R_{\geq0}\cup(0,\rho_{\mathfrak a})\,$ and the equality $\ld(C,{\pi_{\mathbf{1},\nu}})=-\nu^2+\rho_{\mathfrak a}^2=\ld$ determines $\nu = \sqrt{\rho_{\mathfrak a}^2-\lambda}$, where $\nu \in i\R_{\geq0}$ if $\ld \ge \rho_{\mathfrak a}^2$ and $\nu \in (0,\rho_{\mathfrak a}]$ otherwise.

\smallskip

Assume now that $0<p\le [\frac n2]$.
For $\lambda=0$ we have
$$
{{\widehat {G}_{\tau_p,0}}} =
    \begin{cases}
      \left\{J_{\sigma_p,\rho_p}, J_{\sigma_{p-1},\rho_{p-1}}\right\}
        &\quad\text{if } p\neq\frac{n}{2},\\[1mm]
      \left\{D_{\frac n2}^+\oplus D_{\frac n2}^-\right\}
        &\quad\text{if } p=\frac{n}{2},
    \end{cases}
$$
therefore
\begin{equation}\label{eq5:ld=0}
d_0(\tau_p,\G) =
\begin{cases}
  n_\G\left(J_{\sigma_p, \rho_p}\right) + n_\G\left(J_{\sigma_{p-1}, \rho_{p-1}}\right)
    &\text{if } p\neq\frac{n}{2},\\[1mm]
  n_\G\left(D_\frac n2^+ \oplus D_\frac n2^-\right)
    &\text{if } p=\frac{n}{2}.
\end{cases}
\end{equation}

\smallskip

Now, let $\ld>0$.
Since $\ld(C,{\pi_{\sigma_p,\nu}})=-\nu^2+\rho_p^2=\ld$, then $\nu= \sqrt{\rho_p^2-\lambda}$ where $\nu \in (0, \rho_p)\cup i\R_{\geq0}$ and similarly for $\ld(C,{\pi_{\sigma_{p-1},\nu}})=\ld$.
Thus, we get
\begin{equation*}
{\widehat {G}_{\tau_p, \lambda}}  =\left\{ \pi_{\sigma_p,\sqrt {\rho_{p}^2-\lambda}} ,\; \pi_{\sigma_{p-1},\sqrt {\rho_{p-1}^2-\lambda }}\right\}
\end{equation*}
and
\begin{equation}\label{eq5:ld>0}
d_\lambda (\tau_p,\Gamma)=
\begin{cases}
n_\G\left(\pi_{\sigma_p,\sqrt{\rho_{p}^2-\lambda}}\right) + n_\G\left(\pi_{\sigma_{p-1},\sqrt{\rho_{p-1}^2-\lambda}}\right)
    &\text{if }p\neq\frac{n}{2},\\[2mm]
n_\G\left(\pi_{\sigma_{m},\sqrt{1/4-\lambda}}\right) +
n_\G\left(\pi_{\sigma_{m-1},\sqrt{1/4-\lambda}}\right)
    &\text{if }p=\frac{n}{2}=m.\\
\end{cases}
\end{equation}
This completes the proof for $p\le [\frac n2]$.
The case $p>[\frac n2]$ is similar.
\end{proof}

The following lemma is the analogue of Lemma~\ref{lem4:p-1-equiv,p-iso0>p-equiv} in the flat case.

\begin{lemma}\label{lem5:tau_p-1,p-iso=>tau_p}
Let $\G_1$ and $\G_2$ be discrete cocompact subgroups of $\SO(n,1)$ acting freely on $\Hy^n$.
If $\G_1$ and $\G_2$ are $\tau_{p-1}$-equivalent (or $\tau_{p+1}$-equivalent) and the manifolds $\G_1\ba \Hy^n$ and $\G_2\ba \Hy^n$ are $p$-isospectral, then $\G_1$ and $\G_2$ are $\tau_p$-equivalent.
In particular, $0$-isospectrality implies $\tau_0$-equivalence.
\end{lemma}
\begin{proof}
Assume that $p\not \in\{\frac n2, \frac n2+1\}$.
Since $\G_1$ and $\G_2$ are $\tau_{p-1}$-equivalent, we have
\begin{align*}
n_{\G_1}\left(J_{\sigma_{p-1},\rho_{p-1}}\right)
    & = n_{\G_2}\left(J_{\sigma_{p-1},\rho_{p-1}}\right)\\
n_{\G_1}\left(\pi_{\sigma_{p-1},\nu}\right)
    & = n_{\G_2}\left(\pi_{\sigma_{p-1},\nu}\right)
\end{align*}
for every $\nu\in i\R_{\geq0}\cup (0,\rho_{p-1})$ by Proposition~\ref{prop5:hatGtau_neg}.
Now, by $p$-isospectrality we have that $d_\ld(\tau_p,\G_1)=d_\ld(\tau_p,\G_2)$ for every $\ld$, thus \eqref{eq5:ld=0} implies that $n_{\G_1}\left(J_{\sigma_{p},\rho_{p}}\right) = n_{\G_2}\left(J_{\sigma_{p},\rho_{p}}\right)$ and \eqref{eq5:ld>0} implies $n_{\G_1}\left(\pi_{\sigma_{p},\nu}\right) = n_{\G_2}\left(\pi_{\sigma_{p},\nu}\right)$ for every $\nu\in i\R_{\geq0}\cup (0,\rho_{p})$.
By Proposition~\ref{prop5:hatGtau_neg}, these equations imply $\tau_p$-equivalence.

The remaining cases are proved similarly.
 \end{proof}

\begin{proof}[Proof of Theorem~\ref{thm0:A} (noncompact case)]
The proof is exactly as in the flat case (see page \pageref{proof4:thm0}), since Lemma~\ref{lem4:p-1-equiv,p-iso0>p-equiv} and Lemma~\ref{lem5:tau_p-1,p-iso=>tau_p} have exactly the same statements.
\end{proof}

\begin{rem}
One can also prove the above result for intervals decreasing from $n$, that is:
\emph{$q$-isospectrality for every $p\leq q\leq n$ is equivalent to $\tau_q$-equivalence for every $p\leq q\leq n$.}
\end{rem}

\begin{rem}\label{rmk4:Hodge}
We now consider the Hodge decomposition of compact hyperbolic manifolds as in Remark~\ref{rmk3:Hodge} and Remark~\ref{rmk4:Hodge}.
One obtains here results that are very similar to those in the flat case.
Namely
\begin{quote}
  $\Gamma_1\ba \Hy^n$ and $\Gamma_2\ba \Hy^n$ are isospectral on closed (resp.\ coclosed) $p$-forms if and only if $n_{\Gamma_1}(\pi_{\sigma_{p},\nu}) = n_{\Gamma_2}(\pi_{\sigma_{p},\nu})$ (resp.\ $n_{\Gamma_1}(\pi_{\sigma_{p-1},\nu}) = n_{\Gamma_2}(\pi_{\sigma_{p-1},\nu})$) for every $\nu\in i\R_{\geq0}\cup(0,\rho_p)$ (resp.\ $\nu\in i\R_{\geq0}\cup(0,\rho_{p-1})$).
\end{quote}
\end{rem}

\bibliographystyle{plain}

\end{document}